\documentclass[graybox,envcountsect]{svmult}

\usepackage{mathptmx}       % selects Times Roman as basic font
\usepackage{helvet}         % selects Helvetica as sans-serif font
\usepackage{courier}        % selects Courier as typewriter font
\usepackage{type1cm}        % activate if the above 3 fonts are
\usepackage{makeidx}         % allows index generation
\usepackage{graphicx}        % standard LaTeX graphics tool
\usepackage{multicol}        % used for the two-column index
\usepackage[bottom]{footmisc}% places footnotes at page bottom

\usepackage{amsfonts,amsmath}

\makeindex             % used for the subject index
\def \Z {{\mathbb Z}}
\def \R {{\mathbb R}}

\def \N {{\mathbb N}}

\def \da {\downarrow}

\def \PP {{\mathbb P}}
\def \EE {{\mathbb E}}

\def \ES {{\rm E}}
\def \PS {{\rm P}}

\def \w {{\rm (w)}}
\def \b {{\rm (b)}}
\def \n {{(n)}}

\def \cF {{\mathcal F}}

\def \cI {{\mathcal I}}

\def \cB {{\mathcal B}}
\def \cC {{\mathcal C}}

\def\one{\rlap{\mbox{\small\rm 1}}\kern.15em 1}

\begin{document}

\title*{Quenched Lyapunov exponent for the\\ 
parabolic Anderson model in a\\ 
dynamic random environment}
\titlerunning{Quenched Lyapunov exponent for the PAM}
\author{J.\ G\"artner, F.\ den Hollander and G.\ Maillard}
\institute{J.\ G\"artner 
\at Institut f\"ur Mathematik, Technische Universit\"at Berlin,
Strasse des 17.\ Juni 136, D-10623 Berlin, Germany, \email{jg@math.tu-berlin.de}
\and 
F.\ den Hollander 
\at Mathematical Institute, Leiden University, P.O.\ Box 9512, 2300 RA Leiden, 
The Netherlands, \email{denholla@math.leidenuniv.nl}, and EURANDOM, P.O.\ Box 513, 
5600 MB Eindhoven, The Netherlands
\and 
G.\ Maillard \at CMI-LATP, Universit\'e de Provence, 39 rue F. Joliot-Curie, 
F-13453 Marseille Cedex 13, France, \email{maillard@cmi.univ-mrs.fr},
and EURANDOM, P.O.\ Box 513, 5600 MB Eindhoven, The Netherlands}
\maketitle

%\abstract*{}

\abstract{We continue our study of the parabolic Anderson equation $\partial u/\partial t 
= \kappa\Delta u + \gamma\xi u$ for the space-time field $u\colon\,\Z^d\times [0,\infty)
\to\R$, where $\kappa \in [0,\infty)$ is the diffusion constant, $\Delta$ is the discrete 
Laplacian, $\gamma\in (0,\infty)$ is the coupling constant, and $\xi\colon\,\Z^d\times 
[0,\infty)\to\R$ is a space-time random environment that drives the equation. The solution 
of this equation describes the evolution of a ``reactant'' $u$ under the influence of 
a ``catalyst'' $\xi$, both living on $\Z^d$.
\newline\indent
In earlier work we considered three choices for $\xi$: independent simple random 
walks, the symmetric exclusion process, and the symmetric voter model, all in equilibrium 
at a given density. We analyzed the \emph{annealed} Lyapunov exponents, i.e., the 
exponential growth rates of the successive moments of $u$ w.r.t.\ $\xi$, and showed that 
these exponents display an interesting dependence on the diffusion constant $\kappa$, 
with qualitatively different behavior in different dimensions $d$. In the present paper 
we focus on the \emph{quenched} Lyapunov exponent, i.e., the exponential growth rate 
of $u$ conditional on $\xi$. 
\newline\indent
We first prove existence and derive qualitative properties of the quenched Lyapunov 
exponent for a general $\xi$ that is stationary and ergodic under translations in 
space and time and satisfies certain noisiness conditions. After that we focus on the 
three particular choices for $\xi$ mentioned above and derive some further properties. 
We close by formulating open problems.}

%%%%%%%%%%%%%%%%%%%%%%%%%% SECTION 1 %%%%%%%%%%%%%%%%%%%%%%%%%%%%%

\section{Introduction}
\label{S1}

Section~\ref{S1.1} defines the parabolic Anderson model, Section~\ref{S1.2} introduces 
the quenched Lyapunov exponent, Section~\ref{S1.3} summarizes what is known in the 
literature, Section~\ref{S1.4} contains our main results, while Section~\ref{S1.5} 
provides a discussion of these results and lists open problems.

%%%%%

\subsection{Parabolic Anderson model}
\label{S1.1}

The parabolic Anderson model (PAM) is the partial differential equation
\begin{equation}
\label{pA}
\frac{\partial}{\partial t}u(x,t) = \kappa\Delta u(x,t) + [\gamma\xi(x,t)- \delta]u(x,t),
\qquad x\in\Z^d,\,t\geq 0.
\end{equation}
Here, the $u$-field is $\R$-valued, $\kappa\in [0,\infty)$ is the diffusion
constant, $\Delta$ is the discrete Laplacian acting on $u$ as
\begin{equation}
\label{dL}
\Delta u(x,t) = \sum_{{y\in\Z^d} \atop {\|y-x\|=1}} [u(y,t)-u(x,t)]
\end{equation}
($\|\cdot\|$ is the Euclidian norm), $\gamma\in [0,\infty)$ is the coupling constant, 
$\delta\in[0,\infty)$ is the killing constant, while
\begin{equation}
\label{rf}
\xi = (\xi_t)_{t \geq 0} \mbox{ with } \xi_t = \{\xi(x,t) \colon\,x\in\Z^d\}
\end{equation}
is an $\R$-valued random field that evolves with time and that drives the equation. The 
$\xi$-field provides a dynamic random environment defined on a probability space 
$(\Omega,\cF,\PP)$. As initial condition for (\ref{pA}) we take 
\begin{equation}
\label{ic}
u(x,0) = \delta_{0}(x), \qquad x\in\Z^d.
\end{equation}

One interpretation of (\ref{pA}) and (\ref{ic}) comes from population dynamics. Consider 
a system of two types of particles, $A$ (catalyst) and $B$ (reactant), subject to:
\begin{itemize}
\item{$A$-particles evolve autonomously according to a prescribed dynamics with $\xi(x,t)$ 
denoting the number of $A$-particles at site $x$ at time $t$;}
\item{$B$-particles perform independent random walks at rate $2d\kappa$ and split 
into two at a rate that is equal to $\gamma$ times the number of $A$-particles present 
at the same location;}
\item{$B$-particles die at rate $\delta$;}
\item{the initial configuration of $B$-particles is one particle at site $0$ and no 
particle elsewhere.}
\end{itemize}
Then
\begin{equation}
\label{uint}
\begin{array}{lll}
u(x,t) &=& \hbox{the average number of $B$-particles at site $x$ at time $t$}\\
       && \hbox{conditioned on the evolution of the $A$-particles}.
\end{array}
\end{equation}
It is possible to remove $\delta$ via the trivial transformation $u(x,t) \to u(x,t)
e^{-\delta t}$. In what follows we will therefore put $\delta=0$.

Throughout the paper, $\PP$ denotes the law of $\xi$  and we assume that 
\begin{equation}
\label{staterg}
\begin{aligned}
&\bullet\quad \xi \mbox{ is \emph{stationary} and \emph{ergodic} under translations 
in space and time,}\\
&\quad\,\,\,\,\,\,\,\xi \mbox{ is \emph{not constant } and } \rho=\EE(\xi(0,0))\in\R,
\end{aligned}
\end{equation}
and 
\begin{equation}
\label{lambdabd}
\bullet\quad \forall\,\kappa,\gamma\in [0,\infty)\,\,
\exists\,c=c(\kappa,\gamma)<\infty\colon\,
\EE(\log u(0,t)) \leq ct \,\,\forall\,t\geq 0.
\end{equation}
Three choices of $\xi$ will receive special attention:
\begin{description}[(1)]
\item[(1)]{
\emph{Independent Simple Random Walks} (ISRW) [Kipnis and Landim~\cite{kiplan99}, 
Chapter 1]. Here, $\xi_{t}\in \Omega = (\N\cup\{0\})^{\Z^d}$ and $\xi(x,t)$ represents 
the number of particles at site $x$ at time $t$. Under the ISRW-dynamics particles 
move around independently as simple random walks stepping at rate $1$. We draw 
$\xi_{0}$ according to the equilibrium $\nu_\rho$ with density $\rho\in (0,\infty)$,
which is a Poisson product measure.}
\item[(2)]{
\emph{Symmetric Exclusion Process} (SEP) [Liggett~\cite{lig85}, Chapter VIII].
Here, $\xi_{t}\in \Omega = \{0,1\}^{\mathbb{Z}^d}$ and $\xi(x,t)$ represents the 
presence ($\xi(x,t)=1$) or absence ($\xi(x,t)=0$) of a particle at site $x$ at time 
$t$. Under the SEP-dynamics particles move around independently according to an
irreducible symmetric random walk transition kernel at rate $1$, but subject to the 
restriction that no two particles can occupy the same site. We draw $\xi_{0}$ 
according to the equilibrium $\nu_\rho$ with density $\rho\in (0,1)$, which is 
a Bernoulli product measure.}
\item[(3)]{ 
\emph{Symmetric Voter Model} (SVM) [Liggett~\cite{lig85}, Chapter V]. 
Here, $\xi_{t}\in \Omega\ = \{0,1\}^{\mathbb{Z}^d}$ and $\xi(x,t)$ represents the
opinion of a voter at site $x$ at time $t$. Under the SVM-dynamics each voter 
imposes its opinion on another voter according to an irreducible symmetric 
random walk transition kernel at rate $1$. We draw $\xi_{0}$ according to the 
equilibrium distribution $\nu_{\rho}$ with density $\rho\in (0,1)$, which is 
not a product measure.}
\end{description}

\noindent
\emph{Note:} While ISRW and SEP are conservative and reversible in time, SVM is not.
The equilibrium properties of SVM are qualitatively different for recurrent and transient 
random walk. For \emph{recurrent} random walk all equilibria with $\rho\in(0,1)$ are 
non-ergodic, namely, $\nu_\rho=(1-\rho)\delta_{\{\eta\equiv 0\}}+\rho\delta_{\{\eta\equiv 1\}}$, 
and therefore are precluded by (\ref{staterg}). For \emph{transient} random walk, on 
the other hand, there are ergodic equilibria.

%%%%%

\subsection{Lyapunov exponents}
\label{S1.2}

Our focus will be on the \emph{quenched} Lyapunov exponent, defined by
\begin{equation}
\label{qLyapdef}
\lambda_0 = \lim_{t\to\infty} \frac{1}{t} \log u(0,t) \qquad \xi\text{-a.s.}
\end{equation}
We will be interested in comparing $\lambda_0$ with the \emph{annealed} Lyapunov exponents, 
defined by
\begin{equation}
\label{aLyapdef}
\lambda_p = \lim_{t\to\infty} \frac{1}{t} \log\EE\big([u(0,t)]^p\big)^{1/p},
\qquad p \in \N,
\end{equation}
which were analyzed in detail in our earlier work (see Section~\ref{S1.3}). In 
(\ref{qLyapdef}--\ref{aLyapdef}) we pick $x=0$ as the reference site to monitor
the growth of $u$. However, it is easy to show that the Lyapunov exponents are the 
same at other sites.  

By the Feynman-Kac formula, the solution of (\ref{pA}) reads
\begin{equation}
\label{fey-kac1}
u(x,t) = \ES_{\,x}\left(\exp\left[\gamma\int_0^t \xi\left(X^\kappa(s),t-s\right)\,ds\right]
u\big(X^\kappa(t),0\big)\right),
\end{equation}
where $X^\kappa=(X^\kappa(t))_{t \geq 0}$ is simple random walk on $\Z^d$ stepping at rate 
$2d\kappa$ and $\ES_{\,x}$ denotes expectation with respect to $X^\kappa$ given $X^\kappa(0)
=x$. In particular, for our choice in (\ref{ic}), for any $t>0$ we have
\begin{eqnarray}
\label{lyapdef1}
u(0,t)
&=& \ES_{\,0}
\bigg(\exp\bigg[\gamma\int_0^t \xi\big(X^\kappa(s),t-s\big)\,ds\bigg]
\delta_0\big(X^\kappa(t)\big)\bigg)\nonumber\\
&=& \ES_{\,0}
\bigg(\exp\bigg[\gamma\int_0^t \xi\big(X^\kappa(s),s\big)\,ds\bigg]
\delta_0\big(X^\kappa(t)\big)\bigg),
\end{eqnarray}
where in the last line we reverse time and use that $X^\kappa$ is a reversible 
dynamics. Therefore, we can define
\begin{equation}
\label{lyapdef2}
\Lambda_0(t) = \frac{1}{t} \log u(0,t)
=\frac{1}{t} \log \ES_{\,0}
\bigg(\exp\bigg[\gamma\int_0^t \xi\big(X^\kappa(s),s\big)\,ds\bigg]
u\big(X^\kappa(t),0\big)\bigg).
\end{equation}
If the last quantity $\xi$-a.s.\ admits a limit as $t\to\infty$, then
\begin{equation}
\label{lyapdef3}
\lambda_0 = \lim_{t\to\infty} \Lambda_0(t)
\qquad\xi\text{-a.s.},
\end{equation}
where the limit is expected to be $\xi$-a.s.\ constant. 

Clearly, $\lambda_0$ is a function of $d$, $\kappa$, $\gamma$ and the 
parameters controlling $\xi$. In what follows, our main focus will be 
on the dependence on $\kappa$, and therefore we will often write $\lambda_0(\kappa)$. 
Note that $p\mapsto\lambda_p(\kappa)$ is non-decreasing for $p\in\N\cup\{0\}$.

\medskip\noindent
\emph{Note:} Conditions (\ref{staterg}--\ref{lambdabd}) imply that the expectations
in (\ref{fey-kac1}--\ref{lyapdef2}) are strictly positive and finite for all $x\in\Z^d$ 
and $t\geq 0$, and that $\lambda_0<\infty$. Moreover, by Jensen's inequality applied 
to (\ref{lyapdef2}) with $u(\cdot,0)$ given by (\ref{ic}), we have $\EE(\Lambda_0(t))
\geq\rho\gamma+\frac{1}{t}\log \PS_0(X^\kappa(t)=0)$ and, since the last term tends 
to zero as $t\to\infty$, we find that $\lambda_0\geq\rho\gamma>-\infty$. 

%%%%%

\subsection{Literature}
\label{S1.3}

The behavior of the Lyapunov exponents for the PAM in a \emph{time-dependent} random 
environment has been the subject of several papers.

%%%%%

\subsubsection{White noise}
\label{S1.3.1}

Carmona and Molchanov~\cite{carmol94} obtained a qualitative description of both the 
\emph{quenched} and the \emph{annealed} Lyapunov exponents when $\xi$ is white noise, i.e.,
\begin{equation}
\label{whitenoisepot}
\xi(x,t) = \frac{\partial}{\partial t}\,W(x,t),
\end{equation}
where $W=(W_t)_{t \geq 0}$ with $W_t=\{W(x,t)\colon\,x\in\Z^d\}$ is a space-time field 
of independent Brownian motions. This choice is special because the increments of $\xi$ 
are \emph{independent in space and time}. They showed that if $u(\cdot,0)$ has compact 
support (e.g.\ $u(\cdot,0)=\delta_{0}(\cdot)$ as in (\ref{ic})), then the quenched Lyapunov 
exponent $\lambda_0(\kappa)$ defined in (\ref{qLyapdef}) exists and is constant $\xi$-a.s., 
and is independent of $u(\cdot,0)$. Moreover, they found that the asymptotics of 
$\lambda_0(\kappa)$ as $\kappa\downarrow 0$ is singular, namely, there are constants 
$C_1,C_2\in(0,\infty)$ and $\kappa_0\in(0,\infty)$ such that
\begin{equation}
\label{sing1}
C_1\,\frac{1}{\log(1/\kappa)} \leq \lambda_0(\kappa) 
\leq C_2\,\frac{\log\log(1/\kappa)}{\log(1/\kappa)}
\qquad \forall\,0<\kappa \leq \kappa_0.
\end{equation}
Subsequently, Carmona, Molchanov and Viens~\cite{carmolvi96}, Carmona, Koralov and 
Molcha\-nov~\cite{carkormol01}, and Cranston, Mountford and Shiga~\cite{cramoshi02},
proved the existence of $\lambda_0$ when $u(\cdot,0)$ has non-compact support (e.g.\
$u(\cdot,0)\equiv 1$), showed that there is a constant $C\in(0,\infty)$ such that
\begin{equation}
\label{sing2} 
\lim_{\kappa\da 0}\,\log(1/\kappa)\,\lambda_0(\kappa)=C,
\end{equation}
and proved that 
\begin{equation}
\label{plamlim}
\lim_{p \da 0}\,\lambda_p(\kappa) = \lambda_0(\kappa) \quad \forall\,\kappa\in [0,\infty).
\end{equation}
(These results were later extended to L\'evy white noise by Cranston, Mountford and 
Shiga~\cite{cramoshi05}, and to colored noise by Kim, Viens and Vizcarra~\cite{kimvieviz08}.)
Further refinements on the behavior of the Lyapunov exponents were conjectured in
Carmona and Molchanov~\cite{carmol94} and proved in Greven and den Hollander~\cite{grehol07}.
In particular, it was shown that $\lambda_1(\kappa)=\tfrac12$ for all $\kappa\in[0,\infty)$, 
while for the other Lyapunov exponents the following dichotomy holds (see 
Figs.~\ref{fig-lambda01}--\ref{fig-lambda02}):
\begin{itemize}
\item
$d=1,2$: $\lambda_0(\kappa)<\tfrac12$, $\lambda_p(\kappa)>\tfrac12$ 
for $p\in\N\backslash\{1\}$, for $\kappa\in[0,\infty)$;
\item 
$d \geq 3$: there exist $0<\kappa_0 \leq \kappa_2 \leq \kappa_3 \leq \ldots < \infty$ such that 
\begin{equation}
\label{Lyaref0}
\lambda_0(\kappa) - \tfrac12 \left\{
\begin{array}{ll}
< 0, &\mbox{for } \kappa \in [0,\kappa_0),\\
= 0, &\mbox{for } \kappa \in [\kappa_0,\infty),
\end{array}
\right.
\end{equation}
and
\begin{equation}
\label{Lyaref}
\lambda_p(\kappa) - \tfrac12 \left\{
\begin{array}{ll}
> 0, &\mbox{for } \kappa \in [0,\kappa_p),\\
= 0, &\mbox{for } \kappa \in [\kappa_p,\infty),
\end{array}
\right. \qquad p\in\N\backslash\{1\}. 
\end{equation}
\end{itemize}
Moreover, variational formulas for $\kappa_p$ were derived, which in turn led to upper and 
lower bounds on $\kappa_p$, and to the identification of the asymptotics of $\kappa_p$ for 
$p \to \infty$ ($\kappa_p$ grows linearly with $p$). In addition, it was shown that for
every $p\in\N\backslash\{1\}$ there exists a $d(p)<\infty$ such that $\kappa_p<\kappa_{p+1}$ 
for $d\geq d(p)$. Moreover, it was shown that $\kappa_0<\kappa_2$ in Birkner, Greven and 
den Hollander~\cite{birgrehol08} ($d\geq 5$), Birkner and Sun~\cite{birsun10} ($d=4$), 
Berger and Toninelli~\cite{berton09}, Birkner and Sun~\cite{birsun09} ($d=3$). Note that, 
by H\"older's inequality, all curves in Figs.~\ref{fig-lambda01}--\ref{fig-lambda02}
are distinct whenever they are different from $\tfrac12$.

%%%%%%%%%%%%%%%% Begin Figure 1 %%%%%%%%%%%%%%%%%
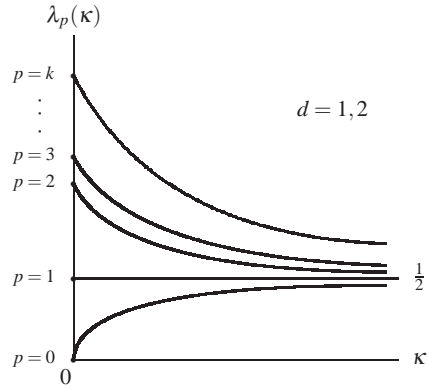
\begin{figure}[htbp]
\vspace{3cm}
\begin{center}
\setlength{\unitlength}{0.18cm}
\begin{picture}(-10,12)(32,-4)
\put(14,-1){\line(18,0){24}}
\put(14,-1){\line(0,24){24}}
{\thicklines
\put(14,5){\line(18,0){24}}
}
{\thicklines
\qbezier(14,-1)(14.2,4.6)(37,4.5)
}
{\thicklines
\qbezier(21,5)(26,5)(37,5)
}
{\thicklines
\qbezier(14,12)(17,5.6)(37,5.5)
}
{\thicklines
\qbezier(14,14)(18,6.6)(37,6)
}
{\thicklines
\qbezier[200](14,20)(20,8.2)(37,7.6)
}
\put(13,-2.8){$0$}
\put(39,-1.3){$\kappa$}
\put(39,4.7){$\frac12$}
\put(12,24){$\lambda_p(\kappa)$}
\put(9.5,-1.2){{\scriptsize$p=0$}}
\put(9.5,4.8){{\scriptsize$p=1$}}
\put(9.5,11.8){{\scriptsize$p=2$}}
\put(9.5,13.8){{\scriptsize$p=3$}}
\put(11.3,15.5){$\cdot$}
\put(11.3,16.8){$\cdot$}
\put(11.3,17.8){$\cdot$}
\put(9.5,19.8){{\scriptsize$p=k$}}
\put(14,-1){\circle*{.35}}
\put(14,5){\circle*{.35}}
\put(14,12){\circle*{.35}}
\put(14,14){\circle*{.35}}
\put(14,20){\circle*{.35}}
\put(30.5,17){{\small $d=1,2$}}
\end{picture}
\caption{\small Quenched and annealed Lyapunov exponents when $d=1,2$ for white noise.}
\label{fig-lambda01}
\end{center}
\end{figure}
%%%%%%%%%%%%%%%%%%%%%%%%%%%%%%%%%%%%%%%%%%%%%%%%%%%%%%%%%%%%%%

%%%%%%%%%%%%%%%% Begin Figure 2 %%%%%%%%%%%%%%%%%%%%%%%%%%%%%%%
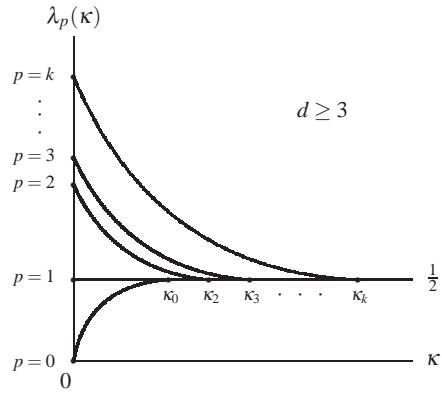
\begin{figure}[htbp]
\vspace{2.2cm}
\begin{center}
\setlength{\unitlength}{0.18cm}
\begin{picture}(60,12)(30,-4)
\put(47,-1){\line(18,0){25}}
\put(47,-1){\line(0,24){24}}
{\thicklines
\put(47,5){\line(18,0){25}}
}
{\thicklines
\qbezier(47,-1)(48,4.6)(54,5)
}
{\thicklines
\qbezier(54,5)(59,5)(71,5)
}
{\thicklines
\qbezier(47,12)(50,5.6)(57,5)
}
{\thicklines
\qbezier(47,14)(51,5.6)(60,5)
}
{\thicklines
\qbezier[200](47,20)(53,5.6)(68,5)
}
\put(46,-3){$0$}
\put(73,-1.3){$\kappa$}
\put(73,4.7){$\frac12$}
\put(45,24){$\lambda_p(\kappa)$}
\put(42.5,-1.4){{\scriptsize$p=0$}}
\put(42.5,4.8){{\scriptsize$p=1$}}
\put(42.5,11.8){{\scriptsize$p=2$}}
\put(42.5,13.8){{\scriptsize$p=3$}}
\put(44.3,15.5){$\cdot$}
\put(44.3,16.8){$\cdot$}
\put(44.3,17.8){$\cdot$}
\put(42.5,19.8){{\scriptsize$p=k$}}
\put(47,-1){\circle*{.35}}
\put(47,5){\circle*{.35}}
\put(47,12){\circle*{.35}}
\put(47,14){\circle*{.35}}
\put(47,20){\circle*{.35}}
\put(54,5){\circle*{.35}}
\put(57,5){\circle*{.35}}
\put(60,5){\circle*{.35}}
\put(68,5){\circle*{.35}}
\put(53.5,3.5){{\scriptsize$\kappa_{0}$}}
\put(56.5,3.5){{\scriptsize$\kappa_{2}$}}
\put(59.5,3.5){{\scriptsize$\kappa_{3}$}}
\put(62,3.5){$\cdot$}
\put(63.5,3.5){$\cdot$}
\put(65,3.5){$\cdot$}
\put(67.5,3.5){{\scriptsize$\kappa_{k}$}}
\put(63.5,17){{\small $d\geq 3$}}
\end{picture}
\caption{\small Quenched and annealed Lyapunov exponents when $d\geq 3$ for white noise.}
\label{fig-lambda02}
\end{center}
\end{figure}
%%%%%%%%%%%%%%%%%%%%%%%%%%%%%%%%%%%%%%%%%%%%%%%%%%%%%%%%%%%%%%%%%%%%

\subsubsection{Interacting particle systems}
\label{S1.3.2}

Various models where $\xi$ is \emph{dependent in space and time} were looked at 
subsequently. Kesten and Sidoravicius~\cite{kessid03}, and G\"artner and den 
Hollander~\cite{garhol06}, considered the case where $\xi$ is a field of independent 
simple random walks in Poisson equilibrium (ISRW). The survival versus extinction 
pattern \cite{kessid03} and the annealed Lyapunov exponents \cite{garhol06} were 
analyzed, in particular, their dependence on $d$, $\kappa$, $\gamma$ and $\rho$. 
The case where $\xi$ is a single random walk was studied by G\"artner and 
Heydenreich~\cite{garhey06}. G\"artner, den Hollander and Maillard~\cite{garholmai07}, 
\cite{garholmai09}, \cite{garholmai10} subsequently considered the cases where $\xi$ 
is an exclusion process with an irreducible symmetric random walk transition kernel 
starting from a Bernoulli product measure (SEP), respectively, a voter model with 
an irreducible symmetric \emph{transient} random walk transition kernel starting either 
from a Bernoulli product measure or from equilibrium (SVM). In each of these cases, 
a fairly complete picture of the behavior of the annealed Lyapunov exponents was 
obtained, including the presence or absence of \emph{intermittency}, i.e., $\lambda_p
(\kappa)>\lambda_{p-1}(\kappa)$ for some or all values of $p\in\N\backslash\{1\}$ and 
$\kappa\in[0,\infty)$. Several conjectures were formulated as well. In what follows 
we describe these results in some more detail. We refer the reader to G\"artner, 
den Hollander and Maillard~\cite{garholmai08HvW} for an overview.

It was shown in G\"artner and den Hollander~\cite{garhol06}, and G\"artner, den Hollander 
and Maillard~\cite{garholmai07}, \cite{garholmai09}, \cite{garholmai10} that for ISRW, 
SEP and SVM in equilibrium the function $\kappa\mapsto\lambda_p(\kappa)$ satisfies: 
\begin{itemize}
\item 
If $d\geq 1$ and $p\in\N$, then the limit in (\ref{aLyapdef}) exists for all $\kappa\in
[0,\infty)$. Moreover, if $\lambda_p(0)<\infty$, then $\kappa\mapsto\lambda_p(\kappa)$
is finite, continuous, strictly decreasing and convex on $[0,\infty)$.
\item 
There are two regimes (we summarize results only for the case where the random walk 
transition kernel has finite second moment):
\begin{itemize}
\item
\emph{Strongly catalytic regime} (see Fig.~\ref{fig-lambda03}): 
\begin{itemize}
\item 
ISRW: $d=1,2$, $p\in\N$ or $d \geq 3$, $p \geq 1/\gamma G_d\colon \lambda_p\equiv \infty$ 
on $[0,\infty)$.\\ 
($G_d$ is the Green function at the origin of simple random walk.)
\item 
SEP: $d=1,2$, $p\in\N\colon \lambda_p\equiv \gamma$ on $[0,\infty)$.
\item 
SVM: $d=1,2,3,4$, $p\in\N\colon \lambda_p\equiv \gamma$ on $[0,\infty)$.
\end{itemize}
\item
\emph{Weakly catalytic regime} (see Fig.~\ref{fig-lambda04}--\ref{fig-lambda05}): 
\begin{itemize}
\item
ISRW: $d \geq 3$, $p<1/\gamma G_d\colon \rho\gamma<\lambda_p< \infty$ on $[0,\infty)$.
\item 
SEP: $d \geq 3$, $p\in\N\colon \rho\gamma<\lambda_p< \gamma$ on $[0,\infty)$.
\item 
SVM: $d \geq 5$,  $p\in\N\colon \rho\gamma<\lambda_p< \gamma$ on $[0,\infty)$.
\end{itemize}
\end{itemize}
\item 
For all three dynamics, in the weakly catalytic regime $\lim_{\kappa\to\infty}
\kappa[\lambda_p(\kappa)-\rho\gamma]=C_1+C_2p^21_{\{d=d_c\}}$ with $C_1,C_2 \in 
(0,\infty)$ and $d_c$ a critical dimension: $d_c=3$ for ISRW, SEP and $d_c=5$ 
for SVM. 
\item 
Intermittent behavior:
\begin{itemize}
\item
In the strongly catalytic regime, there is no intermittency for all three dynamics.
\item
In the weakly catalytic regime, there is full intermittency for:
\begin{itemize}
\item all three dynamics when $0\leq \kappa\ll 1$.
\item ISRW and SEP in $d=3$ when $\kappa\gg 1$.
\item SVM in $d=5$ when $\kappa\gg 1$.
\end{itemize}
\end{itemize}
\end{itemize}

%%%%%%%%%%%%%%%% Begin Figure 3 %%%%%%%%%%%%%%%%%%%%%%%%%%%%%%%%%%%
\begin{figure}[htbp]
\vspace{0.7cm}
\begin{center}
\setlength{\unitlength}{0.20cm}
\begin{picture}(24,12)(0,0)
\put(0,0){\line(18,0){24}}
\put(0,0){\line(0,15){15}}
{\thicklines
\put(0,6){\line(15,0){17}}
\put(0,12){\line(15,0){17}}
}
\qbezier[60](0,2)(9,2)(24,2)
\put(-1,-1.2){$0$}
\put(-2.1,1.7){$\rho\gamma$}
\put(-1.8,5.7){$\gamma$}
\put(-2.1,11.7){$\infty$}
\put(0,6){\circle*{.35}}
\put(0,12){\circle*{.35}}
\put(18,5.7){{\small SEP, SVM}}
\put(18,11.7){{\small ISRW}}
\put(25,-0.3){$\kappa$}
\put(-1.5,16){$\lambda_p(\kappa)$}
\end{picture}
\vspace{.5cm}
\caption{\small Triviality of the annealed Lyapunov exponents for ISRW, SEP, SVM 
in the strongly catalytic regime.}
\label{fig-lambda03}
\end{center}
\end{figure}
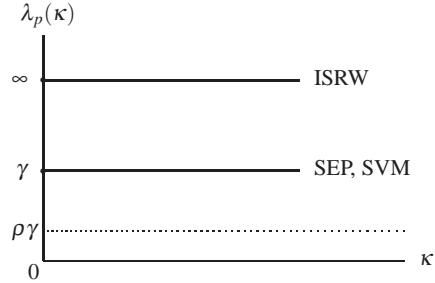
%%%%%%%%%%%%%%%%%%%%%%%%%%%%%%%%%%%%%%%%%%%%%%%%%%%%%%%%%%%%%%%%%%%%%%%%%%%%%

%%%%%%%%%%%%%%%%%%%%%%% Begin Figure 4 %%%%%%%%%%%%%%%%%%%%%%%%%%%%%%%%%%%%%%%
\begin{figure}[htbp]
\begin{center}
\setlength{\unitlength}{0.20cm}
\begin{picture}(-10,10)(13,-1)
\put(0,-4){\line(18,0){18}}
\put(0,-4){\line(0,14){16}}
{\thicklines
\qbezier(0,8)(2,6.6)(4,5.3)
\qbezier(0,6)(2,4.8)(4,3.8)
\qbezier(0,4)(2,3)(4,2.25)
}
{\thicklines
\qbezier(11,2.0)(13,1.4)(17,0.85)
\qbezier(11,1.4)(13,0.9)(17,0.55)
\qbezier(11,0.7)(13,0.4)(17,0.25)
}
\qbezier[80](0,0)(9,0)(18,0)
\put(-2.3,0){$\rho\gamma$}
\put(-1.2,-4.8){$0$}
\put(0,8){\circle*{.35}}
\put(0,6){\circle*{.35}}
\put(0,4){\circle*{.35}}
\put(-3.5,3.8){{\scriptsize$p=1$}}
\put(-3.5,5.8){{\scriptsize$p=2$}}
\put(-3.5,7.8){{\scriptsize$p=3$}}
\put(7,2.5){{\bf ?}}
\put(19,-4.1){$\kappa$}
\put(-1.5,13){$\lambda_p(\kappa)$}
\put(7,10){$d=3$ \,{\small ISRW, SEP}}
\put(7,8.5){$d=5$ \,{\small SVM}}
\end{picture}
\vspace{1.1cm}
\caption{Non-triviality of the annealed Lyapunov exponents for ISRW, SEP and 
SVM in the weakly catalytic regime at the critical dimension.}
\label{fig-lambda04}
\end{center}
\end{figure}
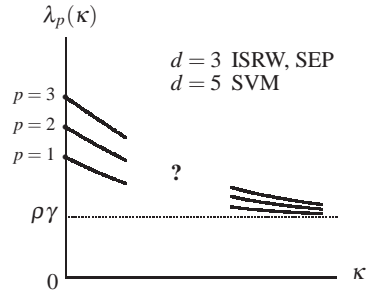
%%%%%%%%%%%%%%%%%%%%%%%%%%%%%%%%%%%%%%%%%%%%%%%%%%%%%%%%%%%%%%%%%%%%%%%%%%%

%%%%%%%%%%%%%%%%%%%%%%% Begin Figure 5 %%%%%%%%%%%%%%%%%%%%%%%%%%%%%%%%%%%%
\begin{figure}[htbp]
\vspace{-0.5cm}
\begin{center}
\setlength{\unitlength}{0.20cm}
\begin{picture}(50,10)(7,1)
\put(24,-4){\line(18,0){18}}
\put(24,-4){\line(0,14){16}}
{\thicklines
\qbezier(24,8)(26,6.6)(28,5.5)
\qbezier(24,6)(26,4.8)(28,3.9)
\qbezier(24,4)(26,3)(28,2.4)
}
\qbezier[25](36,0.8)(38,0.5)(41,0.4)
\qbezier[60](24,0)(33,0)(42,0)
\put(21.7,0){$\rho\gamma$}
\put(22.8,-4.8){$0$}
\put(24,8){\circle*{.35}}
\put(24,6){\circle*{.35}}
\put(24,4){\circle*{.35}}
\put(20.5,3.8){{\scriptsize$p=1$}}
\put(20.5,5.8){{\scriptsize$p=2$}}
\put(20.5,7.8){{\scriptsize$p=3$}}
\put(31,2.5){{\bf ?}}
\put(43,-4.1){$\kappa$}
\put(22.5,13){$\lambda_p(\kappa)$}
\put(31,10){$d\geq 4$ \,{\small ISRW, SEP}}
\put(31,8.5){$d\geq 6$ \,{\small SVM}}
\end{picture}
\vspace{1.4cm}
\caption{Non-triviality of the annealed Lyapunov exponents for ISRW, SEP and 
SVM in the weakly catalytic regime above the critical dimension.} 
\label{fig-lambda05}
\end{center}
\end{figure}
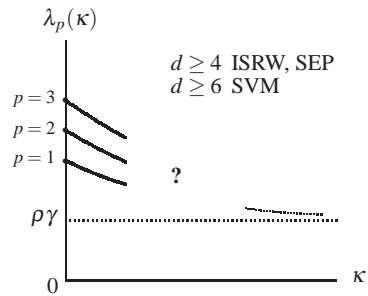
%%%%%%%%%%%%%%%%%%%%%%%%%%%%%%%%%%%%%%%%%%%%%%%%%%%%%%%%%%%%%%%%%%%%%%%%

\noindent
\emph{Note:} For SVM the convexity of $\kappa\mapsto\lambda_p(\kappa)$ and its scaling behavior
for $\kappa\to\infty$ have not actually been proved, but have been argued on heuristic grounds. 

\medskip
Recently, there has been further progress for the case where $\xi$ consists of 
$1$ random walk (Schnitzler and Wolff~\cite{schwol}) or $n$ independent 
random walks (Castell, G\"un and Maillard~\cite{casgunmai11}), $\xi$ is 
the SVM (Maillard, Mountford and Sch\"opfer~\cite{maimousch11}), and for 
the trapping version of the PAM with $\gamma\in(-\infty,0)$ (Drewitz, 
G\"artner, Ram\'{i}rez and Sun~\cite{dregarramsun11}). All these papers 
appear elsewhere in the present volume.

%%%%%%%%%%%%%%%%%%%%%%%%%%%%%%%%%%%%%%%%%%%%%%%%%%%%%%%%%%%%%%%%%%%%%%%%%%%%%%%5

\subsection{Main results}
\label{S1.4}

We have six theorems, all relating to the \emph{quenched} Lyapunov exponent and 
extending the results on the annealed Lyapunov exponents listed in Section~\ref{S1.3}.
 
Let $e$ be any nearest-neighbor site of $0$, and abbreviate
\begin{equation}
\label{Ixidef}
I^\xi(x,t) = \int_0^t [\xi(x,s)-\rho]\,ds, \qquad x \in \Z^d, \,t\geq 0.
\end{equation}
Our first three theorems deal with general $\xi$ and employ four successively stronger 
\emph{noisiness conditions}:
\begin{eqnarray}
\label{E1lim}
&&\lim_{t\to\infty} \frac{1}{\log t}\,\EE\big(|I^\xi(0,t)-I^\xi(e,t)|\big) = \infty,\\
\label{E2E4scal}
&&\liminf_{t\to\infty} \frac{1}{t}\,\EE\big(|I^\xi(0,t)-I^\xi(e,t)|^2\big) > 0, 
\,\,\,\limsup_{t\to\infty} \frac{1}{t^2}\,\EE\big(|I^\xi(0,t)|^4\big) < \infty,\\
\label{expdevweak}
&&\limsup_{t\to\infty} \frac{1}{t^{2/3}}\,\log \Big[\sup_{\eta\in\Omega} 
\PP_{\eta}\big(I^\xi(0,t)>t^{5/6}\big)\Big] < 0,\\
\label{expdev}
&&\exists\,c<\infty\colon\,\,\sup_{\eta\in\Omega} 
\EE_{\eta}\big(\exp\big[\mu I^\xi(0,t)\big]\big) \leq \exp[c\mu^2 t]
\qquad\forall\,\mu,t>0, 
\end{eqnarray}
where $\PP_{\eta}$ denotes the law of $\xi$ starting from $\xi_0=\eta$.

\begin{theorem}
\label{th1}
Fix $d\geq 1$, $\kappa\in[0,\infty)$ and $\gamma\in(0,\infty)$. The limit in 
{\rm (\ref{qLyapdef})} exists $\PP$-a.s.\ and in $\PP$-mean, and is finite.
\end{theorem}

\begin{theorem}
\label{th2}
Fix $d \geq 1$ and $\gamma \in (0,\infty)$.\\
(i) $\lambda_0(0)=\rho\gamma$ and $\rho\gamma<\lambda_0(\kappa)<\infty$ for all 
$\kappa\in(0,\infty)$ with $\rho=\EE(\xi(0,0))\in\R$.\\
(ii) $\kappa\mapsto\lambda_0(\kappa)$ is globally Lipschitz outside any neighborhood 
of $0$. Moreover, if $\xi$ is bounded from above, then the Lipschitz constant at 
$\kappa$ tends to zero as $\kappa\to\infty$.\\
(iii) If $\xi$ satisfies condition {\rm (\ref{E1lim})} and is bounded from below,
then $\kappa\mapsto\lambda_0(\kappa)$ is not Lipschitz at $0$.
\end{theorem}

\begin{theorem}
\label{th3}
(i) If $\xi$ satisfies condition {\rm (\ref{E2E4scal})} and is bounded from below, then
\begin{equation}
\label{slope}
\liminf_{\kappa\da 0}\, \,\log(1/\kappa)\,[\lambda_0(\kappa)-\rho\gamma] > 0.
\end{equation}
(ii) If $\xi$ is a Markov process that satisfies condition {\rm (\ref{expdevweak})} 
and is bounded from above, then 
\begin{equation}
\label{slopealt}
\limsup_{\kappa\da 0}\,\,[\log(1/\kappa)]^{1/6}\,
[\lambda_0(\kappa)-\rho\gamma] < \infty.
\end{equation}
(iii) If $\xi$ is a Markov process that satisfies condition {\rm (\ref{expdev})} 
and is bounded from above, then 
\begin{equation}
\label{cont}
\limsup_{\kappa\da 0}\,\,\frac{\log(1/\kappa)}{\log\log(1/\kappa)}\,
[\lambda_0(\kappa)-\rho\gamma] < \infty.
\end{equation}
\end{theorem}

Our last three theorems deal with ISRW, SEP and SVM. 

\begin{theorem}
\label{th4}
For ISRW, SEP and SVM in the weakly catalytic regime, $\lim_{\kappa\to\infty}$
$\lambda_0(\kappa) = \rho\gamma$.
\end{theorem}

\begin{theorem}
\label{th5}
ISRW and SEP satisfy conditions {\rm (\ref{E1lim})} and {\rm (\ref{E2E4scal})}.
\end{theorem}

\begin{theorem}
\label{th6}
For ISRW in the strongly catalytic regime, $\lambda_0(\kappa)<\lambda_1(\kappa)$ for all 
$\kappa\in[0,\infty)$.
\end{theorem}

Theorems~\ref{th1}--\ref{th3} wil be proved in Section~\ref{S2}, Theorems~\ref{th4}--\ref{th6} 
in Section~\ref{S3}.  

\noindent
\emph{Note:}  Theorem \ref{th4} extends to voter models that are non necessarily 
symmetric (see Section \ref{S3.1}).

%%%%%%%%%%%%%%%%%%%%

\subsection{Discussion and open problems}
\label{S1.5} 

{\bf 1.} Fig.~\ref{fig-lambda06} graphically summarizes the results in Theorems~\ref{th1}--\ref{th3} 
and what we expect to be provable with al little more effort. The main message of this figure is 
that the \emph{qualitative} behavior of $\kappa\mapsto\lambda_0(\kappa)$ is well understood, 
including the logarithmic singularity at $\kappa=0$. Note that Theorems~\ref{th2} and \ref{th3}(i) 
do not imply continuity at $\kappa=0$, while Theorems~\ref{th3}(ii--iii) do. 

%%%%%%%%%%%%%%%%%%%%%%%%%%% Begin Figure 6 %%%%%%%%%%%%%%%%%%%%%%%%%%%%%%%
\begin{figure}[htbp]
\vspace{1cm}
\begin{center}
\setlength{\unitlength}{0.15cm}
\begin{picture}(20,12)(0,-2)
\put(0,0){\line(18,0){22}}
\put(0,0){\line(0,15){15}}
{\thicklines
\qbezier(0,2)(.4,8.8)(3,8)
\qbezier(3,8)(3.8,7.8)(5,6.6)
\qbezier(5,6.6)(9,2.6)(21,2.3)
}
\qbezier[60](0,2)(9,2)(21,2)
\put(-1.4,-1.2){$0$}
\put(-3.2,1.7){$\rho\gamma$}
\put(0,2){\circle*{.35}}
\put(23,-0.3){$\kappa$}
\put(-1.5,16){$\lambda_0(\kappa)$}
\end{picture}
\caption{\small Conjectured behavior of the quenched Lyapunov exponent.}
\label{fig-lambda06}
\end{center}
\end{figure}
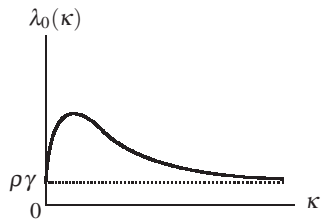
%%%%%%%%%%%%%%%%%%%%%%%%%%%%%%%%%%%%%%%%%%%%%%%%%%%%%%%%%%%%%%%%%%%%%%%

\medskip\noindent
{\bf 2.} 
Figs.~\ref{fig-lambda07}--\ref{fig-lambda09} summarize how we expect 
$\kappa\mapsto\lambda_0(\kappa)$ to compare with $\kappa\mapsto\lambda_1
(\kappa)$ for the three dynamics.

%%%%%%%%%%%%%%%%%%%%%%%%%%% Begin Figure 7 %%%%%%%%%%%%%%%%%%%%%%%%%%%%%%%
\begin{figure}[htbp]
\vspace{1.3cm}
\begin{center}

\setlength{\unitlength}{0.15cm}
\begin{picture}(20,12)(0,-2)
\put(0,0){\line(18,0){22}}
\put(0,0){\line(0,15){15}}
{\thicklines
\qbezier(0,2)(.4,8.8)(3,8)
\qbezier(3,8)(3.8,7.8)(5,6.6)
\qbezier(5,6.6)(9,2.6)(21,2.3)
\qbezier(0,12)(10,12)(21,12)
}
\qbezier[60](0,2)(9,2)(21,2)
\put(-1.4,-1.2){$0$}
\put(-4.6,1.7){{\scriptsize$p=0$}}
\put(-4.6,12){{\scriptsize$p=1$}}
\put(0,2){\circle*{.35}}
\put(23,-0.3){$\kappa$}
\put(-1.5,16){$\lambda_0(\kappa)$}
\end{picture}
\caption{\small Conjectured behavior for ISRW, SEP and SVM below the 
critical dimension.}
\label{fig-lambda07}
\end{center}
\end{figure}
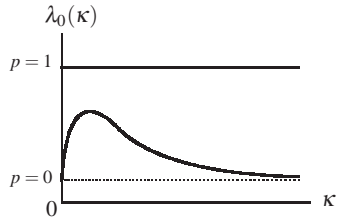
%%%%%%%%%%%%%%%%%%%%%%%%%%%%%%%%%%%%%%%%%%%%%%%%%%%%%%%%%%%%%%%%%%%%%%%

%%%%%%%%%%%%%%%%%%%%%%% Begin Figure 8 %%%%%%%%%%%%%%%%%%%%%%%%%%%%%%%%%
\begin{figure}[htbp]
\vspace{1.3cm}
\begin{center}
\setlength{\unitlength}{0.15cm}
\begin{picture}(20,12)(-18,-2)
\put(-16,0){\line(18,0){22}}
\put(-16,0){\line(0,16){16}}
{\thicklines
\qbezier(-16,2)(-15.6,8.8)(-13,8)
\qbezier(-13,8)(-12.2,7.8)(-11,6.6)
\qbezier(-11,6.6)(-7,2.6)(5,2.3)
}
{\thicklines
\qbezier(-16,14)(-10,3.2)(5,2.8)
}
\qbezier[60](-16,2)(-7,2)(5,2)
\put(-17.6,-1.2){$0$}
\put(-20.5,1.7){{\scriptsize$p=0$}}
\put(-20.5,13.7){{\scriptsize$p=1$}}
\put(-16,2){\circle*{.35}}
\put(-16,14){\circle*{.35}}
\put(7,-0.3){$\kappa$}
\put(-17.5,17){$\lambda_p(\kappa)$}
\put(-10,14){$d=3$ \,{\small ISRW, SEP}}
\put(-10,12.1){$d=5$ \,{\small SVM}}
\end{picture}
\caption{\small Conjectured behavior for ISRW, SEP and SVM at the 
critical dimension.}
\label{fig-lambda08}
\end{center}
\end{figure}
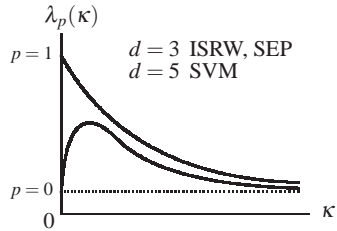
%%%%%%%%%%%%%%%%%%%%%%%%%%%%%%%%%%%%%%%%%%%%%%%%%%%%%%%%%%%%%%%%%%%%%%%%%%%%%%%

%%%%%%%%%%%%% Begin Figure 9 %%%%%%%%%%%%%%%%%%%%%%%%%%%%%%%%%%%%%%%%%%%%%%%%%%
\begin{figure}[htbp]
\vspace{1.3cm}
\begin{center}
\setlength{\unitlength}{0.15cm}
\begin{picture}(20,12)(12,-2)
\put(14,0){\line(18,0){22}}
\put(14,0){\line(0,16){16}}
{\thicklines
\qbezier(14,2)(14.4,8.8)(17,8)
\qbezier(17,8)(17.8,7.8)(19,6.6)
\qbezier(19,6.6)(23,2.6)(35,2.3)
}
{\thicklines
\qbezier(14,14)(20,2.6)(35,2.3)
}
\qbezier[60](14,2)(23,2)(35,2)
\put(12.6,-1.2){$0$}
\put(9.5,1.7){{\scriptsize$p=0$}}
\put(9.5,13.7){{\scriptsize$p=1$}}
\put(31.5,-1.4){$\kappa_0$}
\qbezier[20](32,0)(32,1)(32,2.5)
\put(14,2){\circle*{.35}}
\put(14,14){\circle*{.35}}
\put(37,-0.3){$\kappa$}
\put(12.5,17){$\lambda_p(\kappa)$}
\put(20,14){$d\geq 4$ \,{\small ISRW, SEP}}
\put(20,12.1){$d\geq 6$ \,{\small SVM}}
\end{picture}
\caption{\small Conjectured behavior for ISRW, SEP and SVM above the 
critical dimension.}
\label{fig-lambda09}
\end{center}
\end{figure}
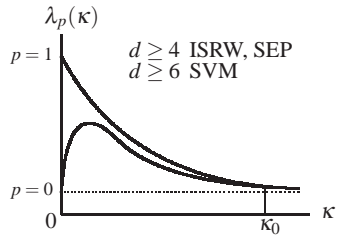
%%%%%%%%%%%%%%%%%%%%%%%%%%%%%%%%%%%%%%%%%%%%%%%%%%%%%%%%%%%%%%%%%%%%%%%%55

\medskip\noindent
{\bf 3.} Conditions (\ref{staterg}--\ref{lambdabd}) are trivially satisfied 
for SEP and SVM, because $\xi$ is bounded. For ISRW they follow from Kesten
and Sidoravicius~\cite{kessid03}, Theorem 2.

\medskip\noindent
{\bf 4.} Conditions (\ref{E1lim}--\ref{E2E4scal}) are \emph{weak} while conditions 
(\ref{expdevweak}--\ref{expdev}) are \emph{strong}. Theorem~\ref{th5} states that
conditions (\ref{E1lim}--\ref{E2E4scal}) are satisfied for ISRW and SEP. We will
see in Section~\ref{S3.2} that, most likely, they are satisfied for SVM as well. 
Conditions (\ref{expdevweak}--\ref{expdev}) fail for the three dynamics, but are 
satisfied e.g.\ for spin-flip dynamics in the so-called ``$M<\epsilon$ regime'' 
(see Liggett~\cite{lig85}, Section I.3). [The verification of this statement is 
left to the reader.] 

\medskip\noindent
{\bf 5.} The following problems remain open:

\begin{itemize}
\item
Extend Theorem~\ref{th1} to the initial condition $u(\cdot,0)\equiv 1$,
and show that $\lambda_0$ is the same as for the initial condition $u(\cdot,0)
= \delta_0(\cdot)$ assumed in (\ref{ic}). [The proof of Theorem~\ref{th1} in
Section~\ref{S2.1} shows that it is straightforward to do this extension for 
$u(\cdot,0)$ symmetric with bounded support. Recall the remark made prior to
(\ref{lyapdef2}).]
\item
Prove that $\lim_{\kappa\downarrow 0} \lambda_0(\kappa)=\rho\gamma$ and 
$\lim_{\kappa\to\infty} \lambda_0(\kappa)=\rho\gamma$ under conditions
(\ref{staterg}--\ref{lambdabd}) alone. [These limits correspond to time 
ergodicity, respectively, space ergodicity of $\xi$, but are non-trivial 
because they require some control on the fluctuations of $\xi$.]
\item
Prove Theorems~\ref{th2}(ii--iii) without the boundedness assumptions on
$\xi$. Prove Theorem~\ref{th3}(i) under condition (\ref{E1lim}) alone.
[The proof of Theorem~\ref{th2}(iii) in Section~\ref{S2.3} shows that 
$\lambda_0(\kappa)-\rho\gamma$ stays above any positive power of $\kappa$ 
as $\kappa\downarrow 0$.] Improve Theorems~\ref{th3}(ii--iii) by establishing 
under what conditions the upper bounds in (\ref{slopealt}--\ref{cont}) can
be made to match the lower bound in (\ref{slope}). 
\item
Extend Theorems~\ref{th4}--\ref{th6} by proving the qualitative behavior 
for the three dynamics conjectured in Figs.~\ref{fig-lambda07}--\ref{fig-lambda09}. 
[For white noise dynamics the curves successively merge for all $d\geq 3$ (see 
Figs.~\ref{fig-lambda01}--\ref{fig-lambda02}).]
\item
For the three dynamics in the weakly catalytic regime, find the asymptotics of 
$\lambda_0(\kappa)$ as $\kappa\to\infty$ and compare this with the asymptotics 
of $\lambda_p(\kappa)$, $p\in\N$, as $\kappa\to\infty$ 
(see Figs.~{\rm \ref{fig-lambda04}}--{\rm \ref{fig-lambda05}}).
\item
Extend the existence of $\lambda_p$ to all (non-integer) $p>0$, and prove that 
$\lambda_p\da\lambda_{0}$ as $p\da 0$. [For white noise dynamics this extension 
is achieved in (\ref{plamlim}).]
\end{itemize}

%%%%%%%%%%%%%%%%%%%%%% SECTION 2 %%%%%%%%%%%%%%%%%%%%%%%

\section{Proof of Theorems \ref{th1}--\ref{th3}}
\label{S2}

The proofs of Theorems~\ref{th1}--\ref{th3} are given in Sections~\ref{S2.1}, 
\ref{S2.4}--\ref{S2.3} and \ref{S2.5}--\ref{S2.7}, respectively. W.l.o.g.\ we 
may assume that $\rho=\EE(\xi(0,0))=0$, by the remark made prior to conditions 
(\ref{staterg}--\ref{lambdabd}). 

%%%%%%%%%%%%%

\subsection{Proof of Theorem \ref{th1}}
\label{S2.1}

\begin{proof}
Recall (\ref{ic}) and (\ref{lyapdef2}--\ref{lyapdef3}), abbreviate
\begin{equation}
\label{subad3}
\chi(s,t)=\ES_{0}\bigg(\exp\bigg[\gamma\int_0^{t-s} \xi\big(X^\kappa(v),s+v\big)\,dv\bigg]
\delta_0(X^\kappa(t-s))\bigg),
\qquad 0\leq s\leq t<\infty,
\end{equation}
and note that $\chi(0,t)\stackrel{\PP}{=}u(0,t)$. Picking $u \in [s,t]$, inserting $\delta_0
(X^\kappa(u-s))$ under the expectation and using the Markov property of $X^\kappa$ at time 
$u-s$, we obtain
\begin{equation}
\label{subad5}
\chi(s,t)\geq \chi(s,u)\, \chi(u,t),
\qquad 0\leq s\leq u\leq t<\infty.
\end{equation}
Thus, $(s,t)\mapsto \log\chi(s,t)$ is superadditive. By condition (\ref{staterg}), the law 
of $\{\chi(u+s,u+t)\colon 0\leq s\leq t<\infty\}$ is the same for all $u\geq 0$. Therefore
the superadditive ergodic theorem (see Kingman \cite{kin76}) implies that
\begin{equation}
\label{subad7}
\lambda_{0}=\lim_{t\to\infty}\frac1t\log\chi(0,t)
\text{ exists } \PP\text{-a.s.\ and in } \PP\text{-mean}. 
\end{equation}
We saw at the end of Section~\ref{S1.2} that $\lambda_0 \in [0,\infty)$ (because
$\rho=0$). 
\qed
\end{proof}

%%%%%

\subsection{Proof of Theorem \ref{th2}(i)}
\label{S2.4}

\begin{proof}
The fact that $\lambda_0(0)=0$ is immediate from (\ref{lyapdef2}--\ref{lyapdef3}) because 
$\PS_0(X^0(t)=0)=1$ for all $t\geq 0$ and $\int_0^t \xi(0,s)\,ds = o(t)$ $\xi$-a.s.\ as 
$t\to\infty$ by the ergodic theorem (recall condition (\ref{staterg})). We already know
that $\lambda_0(\kappa) \in [0,\infty)$ for all $\kappa\in [0,\infty)$. The proof of the 
strict lower bound $\lambda_0(\kappa)>0$ for $\kappa\in (0,\infty)$ comes in 2 steps.

\medskip\noindent
{\bf 1.} 
Fix $T>0$ and consider the expression
\begin{equation}
\label{qLyaplb3}
\lambda_0=\lim_{n\to\infty}\frac{1}{nT}\,\EE\big(\log u(0,nT)\big)
\qquad \xi\text{-a.s.}
\end{equation}
Partition the time interval $[0,nT)$ into $n$ pieces of length $T$,
\begin{equation}
\label{Ijdef}
\cI_j=[(j-1)T,jT), \qquad j=1,\ldots,n.
\end{equation} 
Use the Markov property of $X^\kappa$ at the beginning of each piece, to obtain
\begin{equation}
\label{qLyaplb5}
\begin{aligned}
&u(0,nT)\\
&=\ES_{\,0}\bigg(\exp\bigg[\gamma\sum_{j=1}^n \int_{\cI_j} \xi\big(X^\kappa(s),s\big)\,ds\bigg]
\delta_{0}(X^\kappa(nT))\bigg)\\
&=\sum_{x_1,\ldots,x_{n-1}\in\Z^d} \prod_{j=1}^n \ES_{\,x_{j-1}}
\bigg(\exp\bigg[\gamma\int_0^T \xi\big(X^\kappa(s),(j-1)T+s\big)\,ds\bigg]
\delta_{x_j}(X^\kappa(T))\bigg)
\end{aligned}
\end{equation}
with $x_0=x_n=0$. Next, for $x,y\in\Z^d$, let $\ES_{\,x,y}^{(T)}$ denote the 
conditional expectation over $X^\kappa$ given that $X^\kappa(0)=x$ and $X^\kappa(T)=y$, 
and abbreviate, for $1 \leq j \leq n$,
\begin{equation}
\label{ESxyrep}
\ES_{x,y}^{(T)}(j) = \ES_{\,x,y}^{(T)}
\bigg(\exp\bigg[\gamma\int_0^T \xi\big(X^\kappa(s),(j-1)T+s\big)\,ds\bigg]\bigg).
\end{equation}
Then we can write
\begin{equation}
\label{qLyaplb7}
\begin{aligned}
&\ES_{\,x_{j-1}}\bigg(
\exp\bigg[\gamma\int_0^T \xi\big(X^\kappa(s),(j-1)T+s\big)\,ds\bigg]
\delta_{x_j}(X^\kappa(T))\bigg)\\
&\qquad = p^\kappa_{T}(x_j-x_{j-1})\,\ES_{\,x_{j-1},x_j}^{(T)}(j),
\end{aligned}
\end{equation}
where we abbreviate $p^\kappa_T(x)=\PP_0(X^\kappa(T)=x)$, $x\in\Z^d$. Combined with 
(\ref{qLyaplb5}), this gives
\begin{equation}
\label{qLyaplb9}
\begin{aligned}
u(0,nT)
&= \sum_{x_{1},\cdots,x_{n-1}\in\Z^d}
\Bigg(\prod_{j=1}^n p^\kappa_{T}(x_j-x_{j-1})\Bigg)\,\,
\Bigg(\prod_{j=1}^n \ES_{\,x_{j-1},x_j}^{(T)}(j)\Bigg)\\
&= p^\kappa_{nT}(0)\,\ES_{\,0,0}^{(nT)}
\Bigg(\prod_{j=1}^n \ES_{\,X^\kappa((j-1)T),X^\kappa(jT)}^{(T)}(j)\Bigg).
\end{aligned}
\end{equation}

\medskip\noindent
{\bf 2.}
To estimate the last expectation in (\ref{qLyaplb9}), abbreviate $\xi_{I}
=(\xi_{t})_{t\in I}$, $I\subset [0,\infty)$, and apply Jensen's inequality
to (\ref{ESxyrep}), to obtain
\begin{equation}
\label{qLyaplb11}
\ES_{\,x,y}^{(T)}(j)
= \exp\bigg[\gamma\int_0^T 
\ES_{\,x,y}^{(T)}\Big(\xi\big(X^\kappa(s),(j-1)T+s\big)\Big)\,ds
+ C_{x,y}\big(\xi_{[(j-1)T,jT]},T\big)\bigg]
\end{equation}
for some $C_{x,y}(\xi_{[(j-1)T,jT]},T)$ that satisfies
\begin{equation}
\label{qLyaplb13}
C_{x,y}\big(\xi_{[(j-1)T,jT]},T\big) > 0 \quad \xi\text{-a.s.} \qquad 
\forall\,x,y\in\Z^d,\,1\leq j\leq n.
\end{equation}
Here, the strict positivity is an immediate consequence of the fact that $\xi$ is not
constant (recall condition (\ref{staterg})) and $u \mapsto e^u$ is strictly convex. 
Combining (\ref{qLyaplb9}--\ref{qLyaplb11}) and using Jensen's inequality again, 
this time w.r.t.\ $\ES_{\,0,0}^{(nT)}$, we obtain
\begin{equation}
\label{qLyaplb15}
\begin{aligned}
&\EE\Big(\log u(0,nT)\Big)\\
&\geq\log p^\kappa_{nT}(0)\\
&\qquad +\EE\Bigg(\ES_{\,0,0}^{(nT)}
\Bigg(\sum_{j=1}^n \ES_{\,X^\kappa((j-1)T),X^\kappa(jT)}^{(T)}
\bigg(\gamma\int_0^T \xi\big(X^\kappa(s),(j-1)T+s\big)\,ds\\
&\qquad\qquad\qquad\qquad\qquad
+C_{X^\kappa((j-1)T),X^\kappa(jT)}\big(\xi_{[(j-1)T,jT]},T\big)\bigg)\Bigg)\Bigg)\\
&=\log p^\kappa_{nT}(0)\\ 
&\qquad + \EE\Bigg(\ES_{\,0,0}^{(nT)}
\Bigg(\sum_{j=1}^n \ES_{\,X^\kappa((j-1)T),X^\kappa(jT)}^{(T)}
\Big(C_{X^\kappa((j-1)T),X^\kappa(jT)}\big(\xi_{[(j-1)T,jT]},T\big)\Big)\Bigg)\Bigg),
\end{aligned}
\end{equation}
where the middle term in the second line vanishes because of condition (\ref{staterg})
and our assumption that $\EE(\xi(0,0))=0$. After inserting the indicator of the event 
$\{X^\kappa((j-1)T)=X^\kappa(jT)\}$ for $1\leq j\leq n$ in the last expectation in 
(\ref{qLyaplb15}), we get
\begin{eqnarray}
\label{qLyaplb17}
&&\EE\Bigg(\ES_{\,0,0}^{(nT)}
\Bigg(\sum_{j=1}^n \ES_{\,X^\kappa((j-1)T),X^\kappa(jT)}^{(T)}
\Big(C_{X^\kappa((j-1)T),X^\kappa(jT)}\big(\xi_{[(j-1)T,jT]},T\big)\Big)\Bigg)\Bigg)\nonumber\\
&&\quad\geq 
\sum_{j=1}^n \sum_{z\in\Z^d}
\frac{p^\kappa_{(j-1)T}(z)\,p^\kappa_{T}(0)\,p^\kappa_{(n-j)T}(z)}{p^\kappa_{nT}(0)}\,
\EE\Big(C_{z,z}\big(\xi_{[(j-1)T,jT]},T\big)\Big)\nonumber\\
&&\quad\geq n\,C_T\,p^\kappa_T(0),
\end{eqnarray}
where we abbreviate
\begin{equation}
\label{qLyaplb19}
C_T = \EE\Big(C_{z,z}\big(\xi_{[(j-1)T,jT]},T\big)\Big)>0,
\end{equation}
note that the latter does not depend on $j$ or $z$, and use that $\sum_{z\in\Z^d} 
p^\kappa_{(j-1)T}(z)p^\kappa_{(n-j)T}(z)$ $= p^\kappa_{(n-1)T}(0) \geq p^\kappa_{nT}(0)$. 
Therefore, combining (\ref{qLyaplb3}) and (\ref{qLyaplb15}--\ref{qLyaplb19}), and 
using that 
\begin{equation}
\label{qLyaplb21}
\lim_{n\to\infty} \frac{1}{nT} \log p^\kappa_{nT}(0)=0,
\end{equation}
we arrive at $\lambda_0 \geq (C_T/T)p^\kappa_T(0)>0$. 
\qed
\end{proof}

%%%%%

\subsection{Proof of Theorem \ref{th2}(ii)}
\label{S2.2}

\begin{proof}
In Step 1 we prove the Lischitz continuity outside any neighborhood of $0$
under the restriction that $\xi\leq 1$. This proof is essentially a copy of 
the proof in G\"artner, den Hollander and Maillard~\cite{garholmai10} of the 
Lipschitz continuity of the annealed Lyapunov exponents when $\xi$ is SVM. In 
Step 2 we explain how to remove the restriction $\xi\leq 1$. In Step 3 we show 
that the Lipschitz constant tends to zero as $\kappa\to\infty$ when $\xi\leq 1$.

\medskip\noindent
{\bf 1.} 
Pick $\kappa_1,\kappa_2\in (0,\infty)$ with $\kappa_1<\kappa_2$ arbitrarily.
By Girsanov's formula, 
\begin{eqnarray}
\label{Gir}
&&\ES_{\,0}\bigg(\exp\bigg[\gamma\int_0^t \xi(X^{\kappa_2}(s),s)\,ds
\bigg]\,\delta_0(X^{\kappa_2}(t))\bigg)\nonumber\\
&&\quad= \ES_{\,0}
\bigg(\exp\bigg[\gamma\int_0^t \xi(X^{\kappa_1}(s),s)\,ds\bigg]\,
\delta_0(X^{\kappa_1}(t))\nonumber\\
&&\qquad\qquad\times
\exp\Big[J(X^{\kappa_1};t)\log(\kappa_2/\kappa_1)-2d(\kappa_2-\kappa_1)t
\Big]\bigg)\nonumber\\
&&\quad= I+II,
\end{eqnarray}
where $J(X^{\kappa_1};t)$ is the number of jumps of $X^{\kappa_1}$ up to time $t$,
$I$ and $II$ are the contributions coming from the events $\{J(X^{\kappa_1};t)\leq
M2d\kappa_2t\}$, respectively, $\{J(X^{\kappa_1};t)> M2d\kappa_2t\}$, with $M>1$ 
to be chosen. Clearly,
\begin{equation}
\label{Iest}
\begin{aligned}
I &\leq \exp\Big[\Big(M2d\kappa_2\log(\kappa_2/\kappa_1)
-2d(\kappa_2-\kappa_1)\Big)t\Big]\\
&\qquad \times \ES_{\,0}\bigg(\exp\bigg[\gamma\int_0^t \xi(X^{\kappa_1}(s),s)\,ds\bigg]\,
\delta_0(X^{\kappa_1}(t))\bigg),
\end{aligned}
\end{equation}
while
\begin{equation}
\label{IIest}
II \leq e^{\gamma t}\,\PS_{0}\Big(J(X^{\kappa_2};t)>M2d\kappa_2t\Big)
\end{equation}
because we may estimate 
\begin{equation}
\label{intxibd}
\int_0^t \xi(X^{\kappa_1}(s),s)\,ds \leq t 
\end{equation}
and 
afterwards use Girsanov's formula in the reverse direction. Since 
$J(X^{\kappa_2};t)=J^*(2d\kappa_2t)$ with $(J^*(t))_{t\geq 0}$ a rate-$1$ Poisson 
process, we have
\begin{equation}
\label{LDPPoi}
\lim_{t\to\infty} \frac{1}{t} \log \PS_{0}\Big(J(X^{\kappa_2};t)>M2d\kappa_2t\Big)
= -2d\kappa_2 \cI(M)
\end{equation}
with
\begin{equation}
\label{IMid}
\cI(M) = \sup_{u\in\R} \big[Mu-\big(e^u-1\big)\big] = M\log M-M+1.
\end{equation}
Recalling (\ref{lyapdef2}--\ref{lyapdef3}), we get from (\ref{Gir}--\ref{LDPPoi})
the upper bound
\begin{equation}
\label{lk2lk1bd}
\lambda_0(\kappa_2) \leq
\big[M2d\kappa_2\log(\kappa_2/\kappa_1)
-2d(\kappa_2-\kappa_1)+\lambda_0(\kappa_1)\big]
\vee \big[\gamma-2d\kappa_2\cI(M)\big].
\end{equation}
On the other hand, estimating $J(X^{\kappa_1};t)\geq 0$ in (\ref{Gir}),
we have
\begin{eqnarray}
\label{Jt0}
&&\ES_{\,0}\bigg(\exp\bigg[\gamma\int_0^t \xi(X^{\kappa_2}(s),s)\,ds\bigg]\,
\delta_0(X^{\kappa_2}(t))\bigg)\nonumber\\
&&\qquad\geq  \exp[-2d(\kappa_2-\kappa_1)t]\,\,
\ES_{\,0}\bigg(\exp\bigg[\gamma\int_0^t 
\xi(X^{\kappa_1}(s),s)\,ds\bigg]\,\delta_0(X^{\kappa_1}(t))\bigg),
\end{eqnarray}
which gives the lower bound
\begin{equation}
\label{lk2lk1lb}
\lambda_0(\kappa_2) \geq -2d(\kappa_2-\kappa_1) + \lambda_0(\kappa_1). 
\end{equation}
Next, for $\kappa \in (0,\infty)$, define
\begin{eqnarray}
\label{derdefs}
D^+\lambda_0(\kappa) 
&=& \limsup_{\epsilon \to 0}\,
\epsilon^{-1} [\lambda_0(\kappa+\epsilon)-\lambda_0(\kappa)],\nonumber\\
D^-\lambda_0(\kappa) 
&=& \liminf_{\epsilon \to 0}\,
\epsilon^{-1} [\lambda_0(\kappa+\epsilon)-\lambda_0(\kappa)],
\end{eqnarray}
where $\epsilon \to 0$ from both sides. Then, in (\ref{lk2lk1bd}) and (\ref{lk2lk1lb}), 
picking $\kappa_1=\kappa$ and $\kappa_2=\kappa+\delta$, respectively, $\kappa_1
=\kappa-\delta$ and $\kappa_2=\kappa$ with $\delta>0$ and letting $\delta
\downarrow 0$, we get
\begin{equation}
\label{D1+ublb}
\begin{aligned}
&D^+\lambda_0(\kappa) \leq (M-1)2d \qquad \forall\,M>1\colon\,
2d\kappa \cI(M)-\gamma \geq 0,\\
&D^-\lambda_0(\kappa) \geq -2d.
\end{aligned}
\end{equation}
(The condition in the first line of (\ref{D1+ublb}) guarantees that the first term in 
the right-hand side of (\ref{lk2lk1bd}) is the maximum because $\lambda_0(\kappa)\geq 0$.)
Since $\lim_{M\to\infty} \cI(M)=\infty$, it follows from (\ref{D1+ublb}) that $D^+\lambda_0$ 
and $D^-\lambda_0$ are bounded outside any neighborhood of $\kappa=0$. 

\medskip\noindent
{\bf 2.} 
It remains to explain how to remove the restriction $\xi\leq 1$. Without this restriction
(\ref{intxibd}) is no longer true, but by the Cauchy-Schwarz inequality we have
\begin{equation}
\label{IICS}
II \leq III \times IV
\end{equation}
with 
\begin{equation}
\label{IIICS}
III= \bigg\{\ES_{\,0}
\bigg(\exp\bigg[2\gamma\int_0^t \xi(X^{\kappa_1}(s),s)\,ds\bigg]\,
\delta_0(X^{\kappa_1}(t))\bigg)\bigg\}^{1/2}
\end{equation}
and
\begin{eqnarray}
\label{IVCS}
IV &=& \bigg\{\ES_{\,0}\bigg(
\exp\Big[2J(X^{\kappa_1};t)\log(\kappa_2/\kappa_1)-4d(\kappa_2-\kappa_1)t\Big]\nonumber\\
&&\qquad\qquad\qquad\qquad\times
\one\{J(X^{\kappa_1};t)> M2d\kappa_2t\}\bigg)\bigg\}^{1/2}\nonumber\\
&=& \exp\Big[\Big(d\kappa_{1}-2d\kappa_{2}+d(\kappa_{2}^2/\kappa_{1})\Big)t\Big]\nonumber\\
&&\quad\times
\bigg\{\ES_{\,0}\bigg(
\exp\bigg[J(X^{\kappa_1};t)\log\bigg(\frac{\kappa_2^2/\kappa_1}{\kappa_{1}}\bigg)
-2d\big(\kappa_2^2/\kappa_{1}-\kappa_1\big)t\bigg]\nonumber\\
&&\qquad\qquad\qquad\qquad\times
\one\{J(X^{\kappa_1};t)> M2d\kappa_2t\}\bigg)\bigg\}^{1/2}\nonumber\\
&=& \exp\Big[\Big(d\kappa_{1}-2d\kappa_{2}+d(\kappa_{2}^2/\kappa_{1})\Big)t\Big]
\,\,\bigg\{\PS_{\,0}\bigg(J\Big(X^{\kappa_{2}^2/\kappa_1};t\Big)
> M2d\kappa_2t\bigg)\bigg\}^{1/2},\nonumber\\
&&
\end{eqnarray}
where in the last line we use Girsanov's formula in the reverse direction (without
$\xi$). By (\ref{lyapdef2}--\ref{lyapdef3}) and condition (\ref{lambdabd}), we have 
$III\leq e^{c_0t}$ $\xi$-a.s.\ for $t\geq 0$ and some $c_0<\infty$. Therefore, 
combining (\ref{IICS}--\ref{IVCS}), we get
\begin{equation}
\label{IIest*}
II\leq \exp\Big[\Big(c_{0}+d\kappa_{1}-2d\kappa_{2}+d(\kappa_{2}^2/\kappa_{1})\Big)t\Big]
\,\,\bigg\{\PS_{\,0}\bigg(J\Big(X^{\kappa_{2}^2/\kappa_1};t\Big)
> M2d\kappa_2t\bigg)\bigg\}^{1/2}
\end{equation}
instead of (\ref{IIest}). The rest of the proof goes along the same lines as in 
(\ref{LDPPoi}--\ref{D1+ublb}).

\medskip\noindent
{\bf 3.} Since $\cI(M)>0$ for all $M>1$, it follows from (\ref{D1+ublb}) that 
$\limsup_{\kappa\to\infty} D^+\lambda_0(\kappa) \leq 0$. To prove that 
$\liminf_{\kappa\to\infty} D^-\lambda_0(\kappa) \geq 0$, we argue as follows.
From (\ref{Gir}) with $\kappa_1=\kappa-\delta$ and $\kappa_2=\kappa$, we get
\begin{equation}
\label{basest}
\begin{aligned}
&\ES_{\,0}\bigg(\exp\bigg[\gamma\int_0^t \xi(X^{\kappa}(s),s)\,ds
\bigg]\,\delta_0(X^{\kappa}(t))\bigg)\\
&= \ES_{\,0}
\bigg(\exp\bigg[\gamma\int_0^t \xi(X^{\kappa-\delta}(s),s)\,ds\bigg]\,
\delta_0(X^{\kappa-\delta}(t))\\
&\qquad\qquad\times \exp\Big[J(X^{\kappa-\delta};t)
\log\big(\tfrac{\kappa}{\kappa-\delta}\big)-2d\delta t\Big]\bigg)\\
&\geq e^{-2d\delta t}\,\left[\ES_{\,0}\bigg(\exp\bigg[p\gamma\int_0^t 
\xi(X^{\kappa-\delta}(s),s)\,ds\bigg]\,
\delta_0(X^{\kappa-\delta}(t))\bigg)\right]^{1/p}\\
&\qquad\qquad\times 
\left[\ES_{\,0}
\bigg(\exp\Big[qJ(X^{\kappa-\delta};t)\
\log\big(\tfrac{\kappa}{\kappa-\delta}\big)\Big]\bigg)\right]^{1/q}\\
&= e^{-2d\delta t} \times I \times II,
\end{aligned}
\end{equation}
where we use the reverse H\"older inequality with $(1/p)+(1/q)=1$ and $-\infty<q<0<p<1$.
By direct computation, we have
\begin{equation}
\label{Jexp}
\ES_{\,0}
\bigg(\exp\Big[qJ(X^{\kappa-\delta};t)\
\log\big(\tfrac{\kappa}{\kappa-\delta}\big)\Big]\bigg)
= \exp\Big[-2d(\kappa-\delta)\big[1-(\tfrac{\kappa}{\kappa-\delta})^q] t\Big]
\end{equation}
and hence
\begin{equation}
\label{comp1}
\frac{1}{\delta t}\,\log\Big(e^{-2d\delta t} \times II\Big) 
= -2d - \frac{2d}{\delta q}(\kappa-\delta)\, 
\big[1-\big(\tfrac{\kappa}{\kappa-\delta}\big)^q\big].
\end{equation}
Moreover, with the help of the additional assumption that $\xi\leq 1$, we
can estimate
\begin{equation}
\label{comp2}
I \geq \exp\big[-\big(\tfrac{1-p}{p}\big)\gamma t\big]\,
\left[\ES_{\,0}\bigg(\exp\bigg[\gamma\int_0^t \xi(X^{\kappa-\delta}(s),s)\,ds
\bigg]\,\delta_0(X^{\kappa-\delta}(t))\bigg)\right]^{1/p}.
\end{equation}
Combining (\ref{basest}) and (\ref{comp1}--\ref{comp2}), we arrive at 
(insert $(1-p)/p=-1/q$)
\begin{equation}
\begin{aligned}
&\frac{1}{\delta t}\,
\bigg[\log\ES_{\,0}\bigg(\exp\bigg[\gamma\int_0^t \xi(X^{\kappa}(s),s)\,ds
\bigg]\,\delta_0(X^{\kappa}(t))\bigg)\\
&\qquad\qquad-\log\ES_{\,0}\bigg(\exp\bigg[\gamma\int_0^t \xi(X^{\kappa-\delta}(s),s)\,ds
\bigg]\,\delta_0(X^{\kappa-\delta}(t))\bigg)\bigg]\\
&\geq -2d - \frac{2d}{\delta q}(\kappa-\delta)\, 
\big[1-\big(\tfrac{\kappa}{\kappa-\delta}\big)^q\big] + \frac{\gamma}{\delta q}\\
&\qquad - \frac{1}{\delta q t}\, \log \ES_{\,0}\bigg(\exp\bigg[\gamma\int_0^t 
\xi(X^{\kappa-\delta}(s),s)\,ds\bigg]\,\delta_0(X^{\kappa-\delta}(t))\bigg).
\end{aligned}
\end{equation}
Let $t\to\infty$ to obtain
\begin{equation}
\frac{1}{\delta}[\lambda_0(\kappa)-\lambda_0(\kappa-\delta)]
\geq -2d - \frac{2d}{\delta q}(\kappa-\delta)\, 
\big[1-\big(\tfrac{\kappa}{\kappa-\delta}\big)^q\big] + \frac{1}{\delta q}\,
[\gamma-\lambda_0(\kappa-\delta)].
\end{equation}
Pick $q=-C/\delta$ with $C\in (0,\infty)$ and let $\delta\downarrow 0$, to obtain
\begin{equation}
D^-\lambda_0(\kappa) \geq -2d + \frac{2d\kappa}{C}\,\big(1-e^{-C/\kappa}\big) 
- \frac{1}{C}\,[\gamma-\lambda_0(\kappa)]. 
\end{equation}
Let $\kappa\to\infty$ and use that $\lambda_0(\kappa) \geq 0$, to obtain
\begin{equation}
\liminf_{\kappa\to\infty} D^-\lambda_0(\kappa) \geq -\frac{\gamma}{C}.
\end{equation}
Finally, let $C\to\infty$ to arrive at the claim.
\qed
\end{proof}

%%%%%

\subsection{Proof of Theorem \ref{th2}(iii)}
\label{S2.3}

\begin{proof}
Since $\xi$ is assumed to be bounded from below, we may take $\xi\geq -1$ w.l.o.g., because 
we can scale $\gamma$. The proof of Theorem~\ref{th2}(iii) is based on the following lemma 
providing a lower bound for $\lambda_0(\kappa)$ when $\kappa$ is small enough. Recall 
(\ref{Ixidef}), and abbreviate
\begin{equation}
\label{E1Tdef}
E_1(T) = \EE\big(|I^\xi(0,T)-I^\xi(e,T)|\big), \qquad T>0.
\end{equation} 

\begin{lemma}
\label{slopelwbdlem}
For $T \geq 1$ and $\kappa\downarrow 0$,
\begin{equation}
\label{slopelwbd1}
\lambda_0(\kappa) \geq -\gamma\,\tfrac{1}{T}-2d\kappa\,\tfrac{T-1}{T} 
+ [1+o_\kappa(1)]\,\tfrac{1}{T}\big[\tfrac{\gamma}{2} E_1(T-1)
-\log(1/\kappa)\big].
\end{equation}
\end{lemma}

\begin{proof}
The proof comes in 2 steps. Recall (\ref{ic}) and (\ref{lyapdef2}--\ref{lyapdef3}), 
and write
\begin{equation}
\label{lambdaTkappa}
\begin{aligned}
&\lambda_0(\kappa)\\
&=\lim_{n\to\infty} \frac{1}{nT+1}\log\ES_{\,0}\bigg(\exp\bigg[
\gamma \int_0^{nT+1} \xi(X^\kappa(s),s)\,ds\bigg]
\,\delta_0\big(X^\kappa(nT+1)\big)\bigg).
\end{aligned}
\end{equation}

\medskip\noindent
{\bf 1.}
Partition the time interval $[0,nT+1)$ as $[\cup_{j=1}^{n+1} \cB_j] \cup 
[\cup_{j=1}^n \cC_j]$ with
\begin{equation}
\label{BiCidef}
\begin{array}{lll}
\cB_j &= \big[(j-1)T,(j-1)T+1\big), &1 \leq j \leq n+1,\\[0.2cm]
\cC_j &= \big[(j-1)T+1,jT\big), &1 \leq j \leq n.
\end{array}
\end{equation}
Let
\begin{equation}
\label{Iidef}
I_j^\xi(x) = \int_{\cC_j} \xi(x,s)\,ds
\end{equation}
and
\begin{equation}
\label{Zidef}
Z_j^\xi = \mathrm{argmax}_{x\in\{0,e\}}\, I_j^\xi(x),
\end{equation}
and define the event
\begin{equation}
\label{Axidef}
A^\xi = \left[\,\bigcap_{j=1}^{n}\Big\{
X^\kappa(t) = Z_j^\xi\,\,\forall\,t \in \cC_j\Big\}\right] \cap
\{X^\kappa(nT+1)=0\}.
\end{equation}
We may estimate
\begin{equation}
\label{slopelwbd5}
\begin{aligned}
&\ES_{\,0}\bigg(\exp\bigg[
\gamma \int_0^{nT+1} \xi(X^\kappa(s),s)\,ds\bigg]\,
\delta_0\big(X^\kappa(nT+1)\big)\bigg)\\
&\qquad\geq \ES_{\,0}\bigg(\exp\bigg[
\gamma \int_0^{nT+1} \xi(X^\kappa(s),s)\,ds\bigg]
\one_{A^\xi}\bigg)\\
&\qquad\geq e^{-\gamma(n+1)}
\,\exp\bigg(\gamma \sum_{j=1}^n \max\big\{I_j^\xi(0),I_j^\xi(e)\big\}\bigg)
\,\PS_{\,0}\big(A^\xi\big).
\end{aligned}
\end{equation}
By the ergodic theorem (recall condition (\ref{staterg})), we have
\begin{equation}
\label{slopelwbd7}
\begin{aligned}
&\sum_{j=1}^n \max\big\{I_j^{\xi}(0),I_j^{\xi}(e)\big\}\\
&= [1+o_n(1)]\,n\,\EE\big(\max\big\{I_{1}^\xi(0),I_{1}^\xi(e)\big\}\big)
\quad \xi\text{-a.s.} \mbox{ as } n\to\infty.
\end{aligned}
\end{equation}
Moreover, we have 
\begin{equation}
\label{slopelwbd9}
\PS_{\,0}\big(A^\xi\big)
\geq \Big(\min\big\{p_1^\kappa(0),p_1^\kappa(e)\big\}\Big)^{n+1}\, 
e^{-2d\kappa n(T-1)}
= \big(p_1^\kappa(e)\big)^{n+1}\,e^{-2d\kappa n(T-1)},
\end{equation}
where in the right-hand side the first term is a lower bound for the probability that 
$X^\kappa$ moves from $0$ to $e$ or vice-versa in time $1$ in each of the time 
intervals $\cB_j$, while the second term is the probability that $X^\kappa$ makes no 
jumps in each of the time intervals $\cC_j$.

\medskip\noindent
{\bf 2.}
Combining (\ref{lambdaTkappa}) and (\ref{slopelwbd5}--\ref{slopelwbd9}), and using that
$p_1^\kappa(e)=\kappa[1+o_\kappa(1)]$ as $\kappa \downarrow 0$, we obtain that
\begin{equation}
\label{slopelwbd11}
\begin{aligned}
&\lambda_0(\kappa) \geq -\gamma\,\tfrac{1}{T}-2d\kappa\,\tfrac{T-1}{T}\\ 
&\qquad\qquad + [1+o_\kappa(1)]\,\frac{1}{T}\,
\Big[\gamma\,\EE\big(\max\big\{I_1^\xi(0),I_1^\xi(e)\big\}\big) - \log(1/\kappa)\Big].
\end{aligned}
\end{equation}
Because $I_1^\xi(0)$ and $I_1^\xi(e)$ have zero mean, we have
\begin{equation}
\label{Emax}
\EE\big(\max\big\{I_1^\xi(0),I_1^\xi(e)\big\}\big)
=\tfrac12\,\EE\big(\,\big|I_1^\xi(0)-I_1^\xi(e)\big|\big).
\end{equation}
The expectation in the right-hand side equals $E_1(T-1)$ because $|\cC_1|=T-1$
(recall (\ref{E1Tdef})), and so we get the claim.
\qed
\end{proof}

Using Lemma~\ref{slopelwbdlem}, we can now complete the proof of Theorem~\ref{th2}(iii).
By condition (\ref{E1lim}), for every $c\in (0,\infty)$ we have $E_1(T) \geq c \log T$ 
for $T$ large enough (depending on $c$). Pick $\chi \in (0,1)$ and $T = T(\kappa) 
= \kappa^{-\chi}$ in (\ref{slopelwbd1}) and let $\kappa \downarrow 0$. Then we obtain
\begin{equation}
\label{lambda0lbres}
\lambda_0(\kappa) \geq [1+o_\kappa(1)]\left\{-\gamma\kappa^\chi-2d\kappa +
[\tfrac12c\gamma\chi-1]\,\kappa^\chi\log(1/\kappa)\right\}. 
\end{equation}
Finally, pick $c$ large enough so that $\tfrac12c\gamma\chi>1$. Then, because $\lambda_0(0)
=0$, (\ref{lambda0lbres}) implies that, for $\kappa \downarrow 0$,
\begin{equation}
\label{lambda0lbresext}
\lambda_0(\kappa) - \lambda_0(0) \geq [1+o_\kappa(1)]\,
[\tfrac12c\gamma\chi-1]\,\kappa^\chi\log(1/\kappa), 
\end{equation}
which shows that $\kappa\mapsto\lambda_0(\kappa)$ is not Lipschitz at $0$.
\qed
\end{proof}

%%%%%

\subsection{Proof of Theorem \ref{th3}(i)}
\label{S2.5}

\begin{proof}
Recall (\ref{Ixidef}) and define
\begin{equation}
\label{EkTdef}
\begin{array}{l}
E_k(T) = \EE\big(|I^\xi(0,T)-I^\xi(e,T)|^k\big),\\
\bar{E}_k(T) = \EE\big(|I^\xi(0,T)|^k\big),
\end{array}
\qquad T>0,\,k\in\N.
\end{equation}
Estimate, for $N>0$,
\begin{equation}
\label{U1}
\begin{aligned}
E_1(T) &=\EE\big(|I^\xi(0,T)-I^\xi(e,T)|\big)\\
&\geq \tfrac{1}{2N}\,\EE\left(|I^\xi(0,T)-I^\xi(e,T)|^2\,
\one_{\{|I^\xi(0,T)| \leq N \mbox{ and } |I^\xi(e,T)| \leq N\}}\right)\\
&= \tfrac{1}{2N}\,\Big[E_2(T) - \EE\left(|I^\xi(0,T)-I^\xi(e,T)|^2\,
\one_{\{|I^\xi(0,T)|>N  \mbox{ or } |I^\xi(e,T)|>N\}}\right)\Big].
\end{aligned}
\end{equation}
By Cauchy-Schwarz,
\begin{equation}
\label{U2}
\begin{aligned}
&\EE\left(|I^\xi(0,T)-I^\xi(e,T)|^2\,
\one_{\{|I^\xi(0,T)|>N  \mbox{ or } |I^\xi(e,T)|>N\}}\right)\\
&\qquad \leq [E_4(T)]^{1/2}\,
\left[\PP\left(|I^\xi(0,T)|>N  \mbox{ or } |I^\xi(e,T)|>N\right)\right]^{1/2}.
\end{aligned}
\end{equation}
Moreover, by condition (\ref{staterg}), $E_4(T) \leq 16 \bar{E}_4(T)$ and
\begin{equation}
\label{U3}
\PP\left(|I^\xi(0,T)|>N  \mbox{ or } |I^\xi(e,T)|>N\right)
\leq \tfrac{2}{N^2}\,\bar{E}_2(T) \leq \tfrac{2}{N^2}\,[\bar{E}_4(T)]^{1/2}.
\end{equation}
By condition (\ref{E2E4scal}), there exist an $a>0$ such that $E_2(T) \geq aT$ 
and a $b<\infty$ such that $\bar{E}_4(T) \leq bT^2$ for $T$ large enough. 
Therefore, combining (\ref{U1}--\ref{U3}) and picking $N=cT^{1/2}$ with $c>0$, 
we obtain
\begin{equation}
\label{E1lowbd}
E_1(T) \geq A\,T^{1/2} \mbox{ with } 
A=\tfrac{1}{2c}\,\left(a-2^{5/2}b^{3/4}\tfrac{1}{c}\right),
\end{equation}
where we note that $A>0$ for $c$ large enough. Inserting this bound into 
Lemma~\ref{slopelwbdlem} and picking $T=T(\kappa) = B[\log(1/\kappa)]^2$
with $B>0$, we find that, for $\kappa\downarrow 0$,
\begin{equation}
\lambda_0(\kappa) \geq C\,[\log(1/\kappa)]^{-1}\,[1+o_\kappa(1)]
\mbox{ with } C=\tfrac{1}{B}\left(\tfrac12\gamma AB^{1/2}-1\right).
\end{equation}
Since $C>0$ for $A>0$ and $B$ large enough, this proves the claim in (\ref{slope}). 
\qed
\end{proof}

%%%%%

\subsection{Proof of Theorem \ref{th3}(iii)}
\label{S2.6}

The proof borrows from Carmona and Molchanov~\cite{carmol94}, Section IV.3. 

\begin{proof}
Recall (\ref{ic}) and (\ref{lyapdef2}--\ref{lyapdef3}), estimate
\begin{equation}
\label{}
\lambda_0(\kappa) \leq \limsup_{n\to\infty}\frac{1}{nT}\log 
\ES_{0}\left(\exp\left[\gamma\int_{0}^{nT} \xi(X^\kappa(s),s)\,ds\right]\right),
\end{equation}
and pick
\begin{equation}
\label{Tdef}
T = T(\kappa) = K\log(1/\kappa),
\qquad K \in (0,\infty),
\end{equation}
where $K$ is to be chosen later. Partition the time interval $[0,nT)$ into
$n$ disjoint time intervals $\cI_j$, $1\leq j\leq n$, defined in (\ref{Ijdef}). 
Let $N_{j}$, $1\leq j\leq n$, be the number of jumps of $X^\kappa$ in the time 
interval $\cI_j$, and call $\cI_j$ \emph{black} when $N_{j}>0$ and \emph{white} 
when $N_{j}=0$. Using Cauchy-Schwarz, we can split $\lambda_0(\kappa)$ into a black 
part and a white part, and estimate
\begin{equation}
\label{quesplit}
\lambda_0(\kappa) \leq \mu_0^\b(\kappa)+\mu_0^\w(\kappa),
\end{equation}
where
\begin{eqnarray}
\label{blackLyap}
\mu_0^\b(\kappa) &= \displaystyle{\limsup\limits_{n\to\infty}\frac{1}{2nT}\log 
\ES_{0}\Bigg(\exp\Bigg[2\gamma\sum\limits_{{j=1}\atop {N_{j}>0}}^{n}
\int\limits_{\cI_j} \xi(X^\kappa(s),s)\,ds\Bigg]\Bigg)},\\
\label{whiteLyap}
\mu_0^\w(\kappa) &= \displaystyle{\limsup\limits_{n\to\infty}\frac{1}{2nT}\log 
\ES_{0}\Bigg(\exp\Bigg[2\gamma\sum\limits_{{j=1}\atop {N_{j}=0}}^{n}
\int\limits_{\cI_j} \xi(X^\kappa(s),s)\,ds\Bigg]\Bigg)}.
\end{eqnarray}

\begin{lemma}
\label{blacklem}
If $\xi$ is bounded from above, then there exists a $\delta>0$ such that
\begin{equation}
\label{mublack}
\limsup_{\kappa\da 0}\,\, (1/\kappa)^\delta\,\mu_0^\b(\kappa)\leq 0.
\end{equation}
\end{lemma}

\begin{lemma}
\label{whitelem}
If $\xi$ satisfies condition {\rm (\ref{expdev})}, then 
\begin{equation}
\label{muwhite}
\limsup_{\kappa\da 0}\,\,\frac{\log(1/\kappa)}{\log\log(1/\kappa)}\,\,
\mu_0^\w(\kappa) <\infty.
\end{equation}
\end{lemma}
Theorem~\ref{th3}(ii) follows from (\ref{quesplit}) and Lemmas~\ref{blacklem}--\ref{whitelem}.
\qed
\end{proof}

We first give the proof of Lemma \ref{blacklem}.

\begin{proof}
Let $N^\b=|\{1\leq j\leq n\colon\,N_j>0\}|$ be the number of black time intervals. Since 
$\xi$ is bounded from above (w.l.o.g.\ we may take $\xi \leq 1$, because we can scale 
$\gamma$), we have
\begin{equation}
\label{ubbl}
\begin{aligned}
&\frac{1}{2nT}\,\log\ES_{0}\Bigg(\exp\Bigg[2\gamma\sum_{{j=1}\atop{N_{j}>0}}^{n}
\int_{\cI_j} \xi(X^\kappa(s),s)\,ds\Bigg]\Bigg)\\
&\qquad\leq \frac{1}{2nT}\,\log\ES_{0} \Big(\exp\big[2\gamma T N^\b\big]\Big)\\
&\qquad= \frac{1}{2T}\,\log\Big[\Big(1-e^{-2d\kappa T}\Big)
e^{2\gamma T}+e^{-2d\kappa T}\Big]\\
&\qquad\leq \frac{1}{2T}\,\log\Big[2d\kappa Te^{2\gamma T}+1\Big]\\
&\qquad\leq \frac{1}{2T}\,2d\kappa Te^{2\gamma T}\\
&\qquad= d\kappa^{1-2\gamma K},
\end{aligned}
\end{equation}
where the first equality uses that the distribution of $N^\b$ is $\mathrm{BIN}(n,
1-e^{-2d\kappa T})$, and the second equality uses (\ref{Tdef}). It follows from 
(\ref{blackLyap}) and (\ref{ubbl}) that $\mu_0^\b(\kappa)\leq d\kappa^{1-2\gamma K}$. 
The claim in (\ref{mublack}) therefore follows by picking $0<K<1/2\gamma$ and letting 
$\kappa\downarrow 0$.
\qed
\end{proof}

We next give the proof of Lemma \ref{whitelem}.

\begin{proof}
The proof comes in 5 steps.

\medskip\noindent
{\bf 1.} We begin with some definitions. Define $\Gamma=(\Gamma_1,\ldots,\Gamma_n)$ 
with
\begin{equation}
\label{Gammajdef}
\Gamma_{j}=
\begin{cases}
\{\Delta_{1},\ldots,\Delta_{N_{j}}\} &\text{if } \cI_j \text{ is black},\\
\emptyset &\text{if } \cI_j \text{ is white},
\end{cases}
\end{equation}
where $\Delta_{1},\ldots,\Delta_{N_{j}}$ are the jumps of $X^\kappa$ in the time interval 
$\cI_j$ (which take values in the unit vectors of $\Z^d$). Next, 
let
\begin{equation}
\label{}
\Psi=\{\chi\colon\,\Gamma=\chi\}, \qquad \chi=(\chi_{1},\ldots,\chi_{n}),
\end{equation}
denote the set of possible outcomes of $\Gamma$. Since $X^\kappa$ is stepping at 
rate $2d\kappa$, the random variable $\Gamma$ has distribution
\begin{equation}
\label{gammadistr}
\PS_0(\Gamma=\chi) = e^{-2d\kappa nT}\prod_{j=1}^n 
\frac{(2d\kappa T)^{n_{j}(\chi)}}{n_{j}(\chi)!}, \qquad \chi\in\Psi,
\end{equation}
with $n_j(\chi)=|\chi_j|=|\{\chi_{j,1},\ldots,\chi_{j,n_j(\chi)}\}|$ the number of
jumps in $\chi_j$. For $\chi\in\Psi$, we define the event 
\begin{equation}
\label{Angamma}
A^\n(\chi;\lambda)
=\Bigg\{\sum_{{j=1}\atop{n_{j}(\chi)=0}}^{n}\int_{\cI_j} \xi(x_{j}(\chi),s)\,ds
\geq \lambda\Bigg\}, \qquad \chi\in \Psi,\,\,\, \lambda>0,
\end{equation}
where $x_j(\chi)=\sum_{i=1}^{j-1} \sum_{k=1}^{n_i(\chi)} \chi_{i,k}$ is the location
of $\chi$ at the start of $\chi_j$, and $\lambda$ is to be chosen later. We further 
define
\begin{equation}
\label{kjgammadef}
k_{l}(\chi) = |\{1\leq j\leq n\colon\,n_{j}(\chi)=l\}|,
\qquad l\geq 0,
\end{equation}
which counts the time intervals in which $\chi$ makes $l$ jumps, and we note that
\begin{equation}
\label{kcond}
\sum_{l=0}^\infty k_{l}(\chi) = n.
\end{equation} 

\medskip\noindent
{\bf 2.}
With the above definitions, we can now start our proof. Fix $\chi\in\Psi$. By 
(\ref{Angamma}) and the exponential Chebychev inequality, we have
\begin{equation}
\label{stodomChe}
\begin{aligned}
\PP\big(A^\n(\chi;\lambda)\big) 
&= \PP\left(\sum_{ {j=1} \atop {n_j(\chi)=0} }^n 
\int_{\cI_j} \xi(x_{j}(\chi),s)\,ds \geq \lambda \right)\\
&\leq \inf_{\mu>0} e^{-\mu\lambda}\,\EE\left(\prod_{ {j=1} \atop {n_j(\chi)=0} }^n 
\exp\left[\mu\int_{\cI_j} \xi(x_{j}(\chi),s)\,ds\right]\right)\\
&\leq \inf_{\mu>0} e^{-\mu\lambda} \left[\sup_{\eta\in\Omega} 
\EE_\eta\left(e^{\mu I^\xi(0,T)}\right)\right]^{k_0(\chi)},
\end{aligned}
\end{equation} 
where in the second inequality we use the Markov property of $\xi$ at the beginning 
of the white time intervals, and take the supremum over the starting configuration 
at the beginning of each of these intervals in order to remove the dependence on 
$\xi_{(j-1)T}$, $1 \leq j\leq n$ with $n_j(\chi)=0$, after which we can use (\ref{staterg}) 
and (\ref{Ixidef}). Next, using condition (\ref{expdev}) and choosing $\mu=b\lambda/k_{0}
(\chi)T$, we obtain from (\ref{stodomChe}) that
\begin{equation}
\label{lambdadevbd}
\PP\big(A^\n(\chi;\lambda)\big)
\leq \exp\bigg[-\frac{b\lambda^2}{k_{0}(\chi)T}\bigg]
\bigg\{\exp\bigg[c\bigg(\frac{b\lambda}{k_{0}(\chi)T}\bigg)^2T\bigg]\bigg\}^{k_{0}(\chi)},
\end{equation}
where $c$ is the constant in condition (\ref{expdev}). (Note that $A^\n(\chi;\lambda)=\emptyset$ 
when $k_0(\chi)=0$, which is a trivial case that can be taken care of afterwards.) By picking 
$b=1/2c$, we obtain
\begin{equation}
\label{lambdadevbdext}
\PP\big(A^\n(\chi;\lambda)\big) 
\leq \exp\bigg[-\frac{1}{4c}\,\frac{\lambda^2}{k_{0}(\chi)T}\bigg].
\end{equation}

\medskip\noindent
{\bf 3.}
Our next step is to choose $\lambda$, namely,
\begin{equation}
\label{lambdadef}
\lambda=\lambda(\chi)=\sum_{l=0}^\infty a_{l}k_{l}(\chi)
\end{equation}
with
\begin{equation}
\label{ajdef}
\begin{array}{lll}
a_{0} &= K^\prime\log\log(1/\kappa), &\quad K^\prime \in (0,\infty),\\
a_{l} &= lK\log(1/\kappa), &\quad l\geq 1,
\end{array}
\end{equation}
where $K$ is the constant in (\ref{Tdef}). It follows from (\ref{lambdadevbdext}) 
after substitution of (\ref{lambdadef}--\ref{ajdef}) that (recall (\ref{Tdef}))
\begin{equation}
\label{Aeventest}
\PP\Big(A^\n(\chi;\lambda)\Big) \leq 
\prod_{l=0}^\infty e^{u_lk_l(\chi)} \qquad \forall\,\chi\in\Psi,
\end{equation}
where we abbreviate
\begin{equation}
u_0 = -\frac{1}{4cT}a_{0}^2, \qquad 
u_l = -\frac{1}{2cT}a_{0}a_{l}=-\frac{1}{2c}a_0l,\quad l \geq 1.
\end{equation}
Summing over $\chi$, we obtain
\begin{equation}
\label{BCeq1}
\begin{aligned}
\sum_{\chi\in\Psi}
\PP\Big(A^\n(\chi;\lambda)\Big)
&\leq \sum_{\chi\in\Psi}
\left(\prod_{l=0}^\infty e^{u_lk_l(\chi)}\right)\\
&= \sum_{ (k_l)_{l=0}^\infty \atop {\sum_{l=0}^\infty k_l=n} } 
\left(\frac{n!}{\prod_{l=0}^\infty k_l!}\right) 
\left(\prod_{l=0}^\infty (2d)^{lk_l}\right)
\left(\prod_{l=0}^\infty e^{u_lk_l}\right)\\
&= \Big(\sum_{l=0}^\infty (2d)^l e^{u_l}\Big)^n,
\end{aligned}
\end{equation}
where we use that for any sequence $(k_l)_{l=0}^\infty$ such that 
$\sum_{l=0}^\infty k_l=n$ (recall (\ref{kcond})) the number of $\chi\in\Psi$ 
such that $k_l(\chi)=k_l$, $l\geq 0$, equals $(n!/\prod_{l=0}^\infty k_l!)
\prod_{l=0}^\infty (2d)^{lk_l}$ (note that there are $(2d)^l$ different 
$\chi_{j}$ with $|\chi_{j}|=l$ for each $1\leq j\leq n$). 

\medskip\noindent
{\bf 4.} By (\ref{Tdef}) and (\ref{ajdef}), $T\to\infty$, $a_{0}\to\infty$ 
and $a_{0}^2/T\downarrow 0$ as $\kappa\da 0$. Hence, for $\kappa \downarrow 0$,
\begin{equation}
\label{BCeq2}
\begin{aligned}
\sum_{l=0}^\infty (2d)^l e^{u_l}
&=\exp\left[-\frac{1}{4cT}a_{0}^2\right]
+ \frac{2d\exp\left[-\frac{1}{2c}a_{0}\right]}{1-2d\exp\left[-\frac{1}{2c}a_{0}\right]}\\
&= 1-[1+o_\kappa(1)]\,\frac{[K^\prime\log\log(1/\kappa)]^2}{8cK\log(1/\kappa)}
+ [1+o_\kappa(1)]\,2d[\log(1/\kappa)]^{-K^\prime/2c}\\ 
&= 1-[1+o_\kappa(1)]\,\frac{[K^\prime\log\log(1/\kappa)]^2}{8cK\log(1/\kappa)} < 1,
\end{aligned}
\end{equation}
where the last equality holds provided we pick $K^\prime>2c$. It follows from 
(\ref{BCeq1}--\ref{BCeq2}) that, for $\kappa\downarrow 0$,
\begin{equation}
\label{ABC}
\sum_{n=1}^\infty \PP\left(\bigcup_{\chi\in\Psi} A^\n(\chi;\lambda)\right) < \infty.
\end{equation}
Hence, recalling (\ref{Angamma}), we conclude that, by the Borel-Cantelli lemma, 
$\xi$-a.s.\ there exists an $n_{0}(\xi)\in\N$ such that, for all $n\geq n_{0}(\xi)$,
\begin{equation}
\label{BCfact}
\sum_{{j=1}\atop{n_{j}=0}}^{n}\int_{\cI_j} \xi(x_{j}(\chi),s)\,ds 
\leq \lambda = \sum_{l=0}^\infty a_{l}k_{l}(\chi) 
\qquad \forall\,\chi\in\Psi.
\end{equation} 

\medskip\noindent
{\bf 5.} 
The estimate in (\ref{BCfact}) allows us to proceed as follows. Combining 
(\ref{gammadistr}), (\ref{kjgammadef}) and (\ref{BCfact}), we obtain, for 
$n\geq n_0(\xi,\delta,\kappa)$,
\begin{equation}
\label{whiteakbd}
\begin{aligned}
&\ES_{0}\Bigg(\exp\Bigg[2\gamma\sum_{{j=1} \atop {N_{j}=0}}^{n}
\int_{\cI_j} \xi(X^\kappa(s),s)\,ds\Bigg]\Bigg)\\
&\qquad\leq e^{-2d\kappa n T}\sum_{\chi\in\Psi}
\left(\prod_{l=0}^\infty \left(\frac{(2d\kappa T)^l}{l!}\right)^{k_l(\chi)}\right)
\left(\prod_{l=0}^\infty e^{2\gamma a_{l}k_{l}(\chi)}\right).
\end{aligned}
\end{equation}
Via the same type of computation as in (\ref{BCeq1}), this leads to
\begin{equation}
\label{}
\begin{aligned}
&\ES_{0}\Bigg(\exp\Bigg[2\gamma\sum_{{j=1} \atop {N_{j}=0}}^{n}
\int_{\cI_j} \xi(X^\kappa(s),s)\,ds\Bigg]\Bigg)\\
&\qquad\leq e^{-2d\kappa n T} \sum_{ {(k_l)_{l=0}^\infty} \atop {\sum_{l=0}^\infty k_l=n} }
\left(\frac{n!}{\prod_{l=0}^{\infty} k_{l}!}\right)
\left(\prod_{l=0}^\infty (2d)^{lk_l}\right) \left(\prod_{l=0}^\infty 
\bigg(\frac{(2d\kappa T)^l}{l!}\,e^{2\gamma a_{l}}\bigg)^{k_{l}}\right)\\
&\qquad = e^{-2d\kappa n T}
\Bigg(\sum_{l=0}^\infty \frac{((2d)^2\kappa T)^l}{l!}\,e^{2\gamma a_{l}} \Bigg)^n.
\end{aligned}
\end{equation}
Hence 
\begin{equation}
\label{whitebd}
\begin{aligned}
&\frac{1}{2nT}\log\ES_{0}\Bigg(\exp\Bigg[2\gamma\sum_{{j=1}\atop{N_{j}=0}}^{n}
\int_{\cI_j} \xi(X^\kappa(s),s)\, ds\Bigg]\Bigg)\\
&\quad\leq -d\kappa+\frac{1}{2T}
\log\Bigg(\sum_{l=0}^\infty \frac{((2d)^2\kappa T)^l}{l!}\,e^{2\gamma a_{l}} \Bigg).
\end{aligned}
\end{equation}
Note that the r.h.s.\ of (\ref{whitebd}) does not depend on $n$. Therefore, letting 
$n\to\infty$ and recalling (\ref{whiteLyap}), we get
\begin{equation}
\label{muwbd}
\mu_0^\w(\kappa) \leq -d\kappa+\frac{1}{2T}
\log\Bigg(\sum_{l=0}^\infty \frac{((2d)^2\kappa T)^l}{l!}\,e^{2\gamma a_{l}} \Bigg).
\end{equation}
Finally, by (\ref{Tdef}) and (\ref{ajdef}),
\begin{equation}
\label{sumbdalt}
\begin{aligned}
\sum_{l=0}^\infty \frac{((2d)^2\kappa T)^l}{l!}e^{2\gamma a_{l}}
&= [\log(1/\kappa)]^{2\gamma K^\prime}
+ \sum_{l=1}^\infty \frac{\big((2d)^2\kappa^{1-2\gamma K}K\log(1/\kappa)\big)^l}{l!}\\
&= [\log(1/\kappa)]^{2\gamma K^\prime} + o_\kappa(1),\qquad \kappa\da 0,
\end{aligned}
\end{equation}
where we recall from the proof of Lemma~\ref{blacklem} that $0<K<1/2\gamma$.
Hence
\begin{equation}
\label{whitebdfin}
\mu_0^\w(\kappa) \leq [1+o_\kappa(1)]\,
\frac{\gamma K^\prime\log\log(1/\kappa)}{K\log(1/\kappa)},
\qquad \kappa\da 0,
\end{equation} 
which proves the claim in (\ref{muwhite}).
\qed
\end{proof} 

%%%%%%%%%%%%%%%%%%%%%%%%%%%%

\subsection{Proof of Theorem~\ref{th3}(ii)}
\label{S2.7}

Theorem~\ref{th3}(ii) follows from (\ref{quesplit}), Lemma~\ref{blacklem} and the 
following modification of Lemma~\ref{whitelem}.

\begin{lemma}
\label{whitelemalt}
If $\xi$ satisfies condition (\ref{expdevweak}) and is bounded from above,
then
\begin{equation}
\label{muwhitealt}
\limsup_{\kappa\downarrow 0}\,\,[\log(1/\kappa)]^{1/6}\,\mu^\w(\kappa) <\infty.
\end{equation}
\end{lemma}

\begin{proof}
Most of Steps 1--5 in the proof of Lemma~\ref{whitelem} can be retained.

\medskip\noindent
{\bf 1.} 
Recall (\ref{Gammajdef}--\ref{kcond}). Let
\begin{equation}
\label{fdelTdef}
f_{\epsilon}(T)=\sup_{\eta\in\Omega} \PP_{\eta}\bigg(\int_{0}^T \xi(0,s)\,ds
> \epsilon T\bigg),
\qquad T>0,\epsilon=T^{-1/6}.
\end{equation}
Since $\xi\leq 1$ w.l.o.g., we may estimate
\begin{equation}
\label{stodom}
\sum_{{j=1}\atop{n_{j}(\chi)=0}}^{n} \int_{\cI_j} \xi(x_{j}(\chi),s)\,ds
\preceq Z T+ (k_{0}(\chi)-Z)\epsilon T,
\end{equation}
where $\preceq$ means ``stochastically dominated by'', and $Z$ is the random variable 
with distribution $\PS^*=\mathrm{BIN}(k_{0}(\chi),f_{\epsilon}(T))$. With the help
of (\ref{stodom}), the estimate in (\ref{stodomChe}) can be replaced by
\begin{equation}
\label{stodomChealt}
\begin{aligned}
\PP\big(A^\n(\chi;\lambda)\big) 
&\leq \PS^*\big(ZT+(k_{0}(\chi)-Z)\epsilon T \geq \lambda\big)\\
&\leq \inf_{\mu>0} e^{-\mu\lambda}\,\ES^*\Big(e^{\mu[ZT+(k_{0}(\chi)-Z)\epsilon T]}\Big)\\
&= \inf_{\mu>0} e^{-\mu\lambda}\,
\Big\{f_{\epsilon}(T)\,e^{\mu T}+[1-f_{\epsilon}(T)]\,e^{\mu\epsilon T}\Big\}^{k_{0}(\chi)}.
\end{aligned}
\end{equation} 
Using condition (\ref{expdev}), which implies that there exists a $C \in (0,\infty)$ such 
that $f_{\epsilon}(T)\leq e^{-C\epsilon^2T}$ for $T$ large enough, and choosing $\mu=C\epsilon^2
\lambda/2k_{0}(\chi) T$, we obtain from (\ref{stodomChealt}) that, for $T$ large enough,
\begin{equation}
\label{lambdadevbdalt}
\begin{aligned}
&\PP\big(A^\n(\chi;\lambda)\big)\\
&\qquad \leq \exp\bigg[-\frac{C\epsilon^2\lambda^2}{2k_{0}(\chi) T}\bigg]
\bigg\{\exp\bigg[\frac{C\epsilon^2\lambda}{2 k_{0}(\chi)}-C\epsilon^2T\bigg]
+\exp\bigg[\frac{C\epsilon^3\lambda}{2 k_{0}(\chi)}\bigg]\bigg\}^{k_{0}(\chi)}.
\end{aligned}
\end{equation}

\medskip\noindent
{\bf 2.}
We choose $\lambda$ as
\begin{equation}
\label{lambdadefalt}
\lambda=\lambda(\chi)=\sum_{l=0}^\infty b_{l}k_{l}(\chi)
\end{equation}
with
\begin{equation}
\label{bjdef}
\begin{array}{lll}
b_{0} &= 2\epsilon K\log(1/\kappa) = 2\epsilon T,\\
b_{l} &= lK\log(1/\kappa) = lT, \quad l\geq 1.
\end{array}
\end{equation}
Note that this differs from (\ref{ajdef}) only for $l=0$, and that (\ref{kcond})
implies, for $T$ large enough, 
\begin{equation}
\label{labdalt}
\lambda \geq n2\epsilon T.
\end{equation}

\medskip\noindent
{\bf 3.}
Abbreviate the two exponentials between the braces in the right-hand side 
of (\ref{lambdadevbd}) by $I$ and $II$. Fix $A \in (1,2)$. In what follows 
we distinguish between two cases: $\lambda>Ak_{0}(\chi)T$ and $\lambda\leq
Ak_{0}(\chi)T$.

\medskip\noindent
\underline{$\lambda>Ak_{0}(\chi)T$}:   
Abbreviate $\alpha_1=\tfrac14A>0$. Neglect the term $-C\epsilon^2T$ in $I$, 
to estimate, for $T$ large enough,
\begin{equation}
\begin{aligned}
I+II &\leq 
\exp\left[\frac{C\epsilon^2\lambda}{2k_{0}(\chi)}\right]\,
\left(1+\exp\left[-\frac{C\epsilon^2\lambda}{4k_{0}(\chi)}\right]\right)\\
&\leq \exp\left[\frac{C\epsilon^2\lambda}{2k_{0}(\chi)}\right]\,\,
\left(1+e^{-\alpha_1 C\epsilon^2T}\right).
\end{aligned}
\end{equation}
This yields
\begin{equation}
\label{Aest1}
\PP\big(A^\n(\chi;\lambda)\big) 
\leq \exp\left[-\frac{C\epsilon^2\lambda^2}{2k_{0}(\chi)T}
+\frac{C\epsilon^2\lambda}{2}\right]\,\exp\left[k_{0}(\chi)e^{-\alpha_1 C\epsilon^2T}\right].
\end{equation}

\medskip\noindent
\underline{$\lambda \leq Ak_{0}(\chi)T$}:
Abbreviate $\alpha_2= 1-\tfrac12 A>0$. Note that $I\leq\exp[-\alpha_2 C\epsilon^2T]$ 
and $II\geq 1$, to estimate
\begin{equation}
I+II \leq II\,\big(1+e^{-\alpha_2 C\epsilon^2T}\big).
\end{equation}  
This yields
\begin{equation}
\label{Aest2}
\PP\big(A^\n(\chi;\lambda)\big) 
\leq \exp\left[-\frac{C\epsilon^2\lambda^2}{2k_{0}(\chi)T}
+\frac{C\epsilon^3\lambda}{2}\right]\,\exp\left[k_{0}(\chi)e^{-\alpha_2 C\epsilon^2T}\right].
\end{equation}

\medskip\noindent
We can combine (\ref{Aest1}) and (\ref{Aest2}) into the single estimate
\begin{equation}
\label{Aest3} 
\PP\big(A^\n(\chi;\lambda)\big) 
\leq \exp\left[-\frac{C'\epsilon^4\lambda^2}{2k_{0}(\chi)T}\right]
\,\exp\left[k_{0}(\chi)e^{-\alpha CT}\right]
\end{equation}
for some $C^\prime=C^\prime \in (0,\infty)$ with $\alpha = \alpha_1 \wedge \alpha_2>0$. To see 
why, put $x=\lambda/k_{0}(\chi)T$, and rewrite the exponent of the first exponential in the 
right-hand side of (\ref{Aest1}) and (\ref{Aest2}) as
\begin{equation}
\tfrac12C\epsilon^2k_{0}(\chi)T\,(-x^2+x), \quad \mbox{ respectively, } \quad
\tfrac12C\epsilon^2k_{0}(\chi)T\,(-x^2+\epsilon x).
\end{equation}
In the first case, since $A>1$, there exists a $B>0$ such that $-x^2+x \leq -Bx^2$ 
for all $x\geq A$. In the second case, there exists a $B>0$ such that $-x^2+\epsilon 
x \leq -B\epsilon^2x^2$ for all $x\geq 2\epsilon$. But (\ref{labdalt}) ensures that
$x\geq 2\epsilon n/k_0(\chi)\geq 2\epsilon$. Thus, we indeed get (\ref{Aest3}) with 
$C'=CB$.

\medskip\noindent
{\bf 4.} 
The same estimates as in (\ref{Aeventest}--\ref{BCeq1}) lead us to
\begin{equation}
\label{BCeq1alt}
\begin{aligned}
\sum_{ {\chi\in\Psi}}
\PP\Big(A^\n(\chi;\lambda)\Big)
\leq \Big(\sum_{l=0}^\infty (2d)^l e^{v_l}\Big)^n.
\end{aligned}
\end{equation}
with
\begin{equation}
v_0 = -\frac{C^\prime\epsilon^4}{2T}b_{0}^2+e^{-\alpha C\epsilon^2T}, \qquad 
v_l = -\frac{C^\prime\epsilon^4}{T}b_{0}b_{l}=-C^\prime\epsilon^4 b_0l,\quad l \geq 1.
\end{equation}
By (\ref{Tdef}) and (\ref{bjdef}), we have
\begin{equation}
\label{BCeq2alt}
\begin{aligned}
\sum_{l=0}^\infty (2d)^l e^{v_l}
&=\exp\left[-\frac{C^\prime\epsilon^4}{2T}b_{0}^2+e^{-\alpha C\epsilon^2T}\right]
+ \frac{2d\exp\left[-C^\prime\epsilon^4 b_{0}\right]}{1-2d\exp\left[-C^\prime\epsilon^4 b_{0}\right]}\\
&= \exp\left[-2C^\prime + e^{-\alpha C T^{2/3}}\right]
+ \frac{2d\exp[-2C^\prime T^{1/6}]}{1-2d\exp[-2C^\prime T^{1/6}]}\\
&< 1
\end{aligned}
\end{equation}
for $T$ large enough, i.e., $\kappa$ small enough. This replaces (\ref{BCeq2}). Therefore the 
analogues of (\ref{ABC}--\ref{BCfact}) hold, i.e., $\xi$-a.s.\ there exists an $n_{0}(\xi)\in\N$ 
such that, for all $n\geq n_{0}(\xi)$,
\begin{equation}
\label{BCfactalt}
\sum_{{j=1}\atop{n_{j}=0}}^{n}\int_{\cI_j} \xi(x_{j}(\chi),s)\,ds 
\leq \lambda = \sum_{l=0}^\infty b_{l}k_{l}(\chi) 
\qquad \forall\,\chi\in\Psi.
\end{equation} 

\medskip\noindent
{\bf 5.} 
The same estimate as in (\ref{whiteakbd}--\ref{muwbd}) now lead us to
\begin{equation}
\mu_0^\w(\kappa) \leq -d\kappa+\frac{1}{2T}
\log\Bigg(\sum_{l=0}^\infty \frac{((2d)^2\kappa T)^l}{l!}\,e^{2\gamma b_{l}} \Bigg).
\end{equation}
Finally, by (\ref{Tdef}) and (\ref{bjdef}),
\begin{equation}
\begin{aligned}
\sum_{l=0}^\infty \frac{((2d)^2\kappa T)^l}{l!}e^{2\gamma b_{l}}
&= e^{4\gamma\epsilon T}
+ \sum_{l=1}^\infty \frac{\big((2d)^2\kappa^{1-2\gamma K}K\log(1/\kappa)\big)^l}{l!}\\
&= e^{4\gamma\epsilon T} + o_\kappa(1),\qquad \kappa\da 0,
\end{aligned}
\end{equation}
which replaces (\ref{sumbdalt}). Hence
\begin{equation}
\label{whitebdfinalt}
\mu_0^\w(\kappa) \leq [1+o_\kappa(1)]\,2\gamma\epsilon, \qquad \kappa\downarrow 0,
\end{equation} 
which proves the claim in (\ref{muwhitealt}).
\qed
\end{proof}

%%%%%%%%%%%%%%%%%%%%%%%%%%%%%%%%%%%%%%%%%%%%%%%%%%%%%%%

\section{Proof of Theorems \ref{th4}--\ref{th6}}
\label{S3}

The proofs of Theorems~\ref{th4}--\ref{th6} are given in Sections~\ref{S3.1}--\ref{S3.3},
respectively.

%%%%%%%%%%%%%%%%%%%%

\subsection{Proof of Theorem \ref{th4}}
\label{S3.1}

\begin{proof}
For ISRW, SEP and SVM in the weakly catalytic regime, it is known that 
$\lim_{\kappa \to\infty} \lambda_1(\kappa) =\rho\gamma$ (recall Section~\ref{S1.3.2}). 
The claim therefore follows from the fact that 
$\rho\gamma\leq\lambda_0(\kappa)\leq \lambda_1(\kappa)$ for all $\kappa\in[0,\infty)$. 

\noindent
\emph{Note:} The claim extends to non-symmetric voter models (see \cite{garholmai10}, 
Theorems 1.4--1.5).
\qed
\end{proof}

%%%%%%%%%%%%%%%%%%%%

\subsection{Proof of Theorem \ref{th5}}
\label{S3.2}

\begin{proof}
It suffices to prove condition (\ref{E2E4scal}), because we saw in Section~\ref{S2.5} that 
condition (\ref{E2E4scal}) implies (\ref{E1lowbd}), which is stronger than condition 
(\ref{E1lim}). Step 1 deals with $E_2(T)$, step 2 with $\bar E_4(T)$. 

\medskip\noindent
{\bf 1.} Let
\begin{equation}
\label{Cxtdef}
C(x,t) = \EE\big([\xi(0,0)-\rho][\xi(x,t)-\rho]\big), \qquad x\in\Z^d, t \geq 0,
\end{equation}
denote the two-point correlation function of $\xi$. By condition (\ref{staterg}), 
we have
\begin{equation}
\label{E2eq}
\begin{aligned}
E_2(T) &= \int_0^T ds \int_0^T dt\,\,\,\EE\big([\xi(0,s)-\xi(e,s)][\xi(0,t)-\xi(e,t)]\big)\\
&= 4 \int_0^T ds \int_0^{T-s} dt\,\,\,[C(0,t)-C(e,t)].
\end{aligned}
\end{equation} 
In what follows $G_d(x)=\int_0^\infty p_t(x)dt$, $x\in\Z^d$, denotes the Green function of 
simple random walk on $\Z^d$ stepping at rate $1$ starting from $0$, which is finite if and 
only if $d \geq 3$. (Recall from Section~\ref{S1.3.2} that SVM with a simple random walk 
transition kernel is of no interest in $d=1,2$.)

\begin{lemma}
\label{2ptcorlem}
For $x\in\Z^d$ and $t\geq0$,
\begin{equation}
\label{Cxtprop}
C(x,t)=
\begin{cases}
\rho p_{t}(x), &\quad \text{\rm ISRW},\\
\rho(1-\rho)p_{t}(x), &\quad \text{\rm SEP},\\
[\rho(1-\rho)/G_{d}(0)]\int_{0}^\infty p_{t+u}(x) \,du,
&\quad \text{\rm SVM}.
\end{cases}
\end{equation}
\end{lemma}
\begin{proof}
For ISRW, we have 
\begin{equation}
\label{ISRWrep}
\xi(x,t) = \sum_{y\in\Z^d}\sum_{j=1}^{N^{y}}\delta_{x}\big(Y_{j}^y(t)\big),
\qquad x\in\Z^d, t\geq 0, 
\end{equation}
where $\{N^{y}\colon y\in\Z^d\}$ are i.i.d.\ Poisson random variables with mean 
$\rho\in(0,\infty)$, and $\{Y_{j}^y\colon y\in\Z^d, 1\leq j\leq N^{y}\}$ with
$Y_j^y=(Y_j^y(t))_{t\geq 0}$ is a collection of independent simple random walks 
with jump rate $1$ ($Y_{j}^y$ is the $j$-th random walk starting from $y\in\Z^d$). 
Inserting (\ref{ISRWrep}) into (\ref{Cxtdef}), we get the first line in (\ref{Cxtprop}). 
For SEP and SVM, the claim follows via the graphical representation (see G\"artner, 
den Hollander and Maillard~\cite{garholmai07}, Eq.\ (1.5.5), and \cite{garholmai10}, 
Lemma A.1, respectively). Recall from the remark made at the end of Section~\ref{S1.1}
that SVM requires the random walk transition kernel to be \emph{transient}.
\qed 
\end{proof}

For ISRW, (\ref{E2eq}--\ref{Cxtprop}) yield
\begin{equation}
\frac{1}{T}\,E_2(T) 
= 4\rho \int_0^T dt\,\frac{T-t}{T}\,[p_t(0)-p_t(e)],
\end{equation}
where we note that $p_t(0)-p_t(e) \geq 0$ by the symmetry of the random walk
transition kernel. Hence, by monotone convergence, 
\begin{equation}
\lim_{T\to\infty} \frac1T E_{2}(T) 
= 4\rho \int_0^\infty dt\,[p_t(0)-p_t(e)],
\end{equation} 
which is a number in $(0,\infty)$ (see Spitzer~\cite{spi64}, Sections 24 and 29). 
For SEP, the same computation applies with $\rho$ replaced by $\rho(1-\rho)$. For 
SVM, (\ref{E2eq}--\ref{Cxtprop}) yield
\begin{equation}
\frac{1}{T}\,E_2(T) 
= 4\frac{\rho(1-\rho)}{G_d(0)} \int_0^\infty du\,
\left[\frac{u(2T-u)}{2T}\,\one_{\{u \leq T\}} 
+ \tfrac12T\,\one_{\{u \geq T\}}\right]\,[p_u(0)-p_u(e)].
\end{equation}
Hence, by monotone convergence (estimate $\tfrac12 T\leq \tfrac12 u$ in the second 
term of the integrand),
\begin{equation}
\lim_{T\to\infty} \frac1T E_{2}(T)
= 4\frac{\rho(1-\rho)}{G_d(0)} \int_0^\infty du\,\,u\,[p_u(0)-p_u(e)],
\end{equation} 
which again is a number in $(0,\infty)$ (see Spitzer~\cite{spi64}, Section 24).

\medskip\noindent
{\bf 2.} Let 
\begin{equation}
\begin{aligned}
&C(x,t;y,u;z,v)=\EE\big([\xi(0,0)-\rho][\xi(x,t)-\rho][\xi(y,u)-\rho][\xi(z,v)-\rho]\big),\\
&x,y,z\in\Z^d,\, 0\leq t\leq u\leq v,
\end{aligned}
\end{equation} 
denote the four-point correlation function of $\xi$. Then, by condition (\ref{staterg}),
\begin{equation}
\label{Ebar4}
\bar{E}_4(T) = 4! \int_0^T ds \int_0^{T-s} dt \int_t^{T-s} du \int_u^{T-s} dv\,\,\,
C(0,t;0,u;0,v).
\end{equation} 
To prove the second part of (\ref{E2E4scal}), we must estimate $C(0,t;0,u;0,v)$.
For ISRW, this can be done by using (\ref{ISRWrep}), for SEP by using the Markov property 
and the graphical representation. In both cases the computations are long but straightforward, 
with leading terms of the form 
\begin{equation}
M p_{a}(0,0)p_{b}(0,0)
\end{equation}
with $a,b$ linear in $t$, $u$ or $v$, and $M<\infty$. Each of these leading terms, after 
being integrated as in (\ref{Ebar4}), can be bounded from above by a term of order 
$T^2$, and hence $\limsup_{T\to\infty} \bar{E}_4(T)/T^2<\infty$. The details are 
left to the reader.
\qed
\end{proof}

\medskip\noindent
\emph{Note:} We expect the second part of condition (\ref{E2E4scal}) to hold also for 
SVM. However, the graphical representation, which is based on coalescing random walks, 
seems too cumbersome to carry through the computations.

%%%%%%%%%%%%%%

\subsection{Proof of Theorem \ref{th6}}
\label{S3.3}

\begin{proof}
For ISRW in the strongly catalytic regime, we know that $\lambda_1(\kappa)=\infty$ 
for all $\kappa\in [0,\infty)$ (recall Fig.~\ref{fig-lambda03}), while $\lambda_0(\kappa)
<\infty$ for all $\kappa\in [0,\infty)$ (by Kesten and Sidoravicius~\cite{kessid03}, 
Theorem 2).
\qed
\end{proof}

\begin{acknowledgement}
GM is grateful to CNRS for financial support and to EURANDOM for hospitality.
We thank Dirk Erhard for his comments on an earlier draft of the manuscript.
\end{acknowledgement}

%%%%%%%%%%%%%%%%%%%%%%%%%%%%%%%%%%%%%%%%%%%%%%%%%%%%%%%%%%%%%%%%

%%%%%%%%%%%%%%%%%%%%%%%% referenc.tex %%%%%%%%%%%%%%%%%%%%%%%%%%%%%%
% sample references
% %
% Use this file as a template for your own input.
%
%%%%%%%%%%%%%%%%%%%%%%%% Springer-Verlag %%%%%%%%%%%%%%%%%%%%%%%%%%
%
% BibTeX users please use
% \bibliographystyle{}
% \bibliography{}
%

\end{document}